\documentclass[12pt,reqno]{amsart}

\newcommand\version{March 19, 2024}

\usepackage{amsmath,amsfonts,amsthm,amssymb,amsxtra}
\usepackage{bbm} 
\usepackage{hyperref} 

\newtheorem{theorem}{Theorem}
\newtheorem{proposition}[theorem]{Proposition}
\newtheorem{lemma}[theorem]{Lemma}
\newtheorem{corollary}[theorem]{Corollary}

\theoremstyle{definition}

\theoremstyle{remark}

\newtheorem{remark}[theorem]{Remark}
\newtheorem{remarks}[theorem]{Remarks}

\newcommand{\1}{\mathbbm{1}}

\newcommand{\const}{\mathrm{const}\ }

\renewcommand{\epsilon}{\varepsilon}

\newcommand{\loc}{{\rm loc}}

\newcommand{\N}{\mathbb{N}}

\renewcommand{\phi}{\varphi}
\newcommand{\R}{\mathbb{R}}

\newcommand{\Z}{\mathbb{Z}}

\DeclareMathOperator{\dist}{dist}

\DeclareMathOperator{\spec}{spec}

\DeclareMathOperator{\sgn}{sgn}
\DeclareMathOperator{\Tr}{Tr}
\DeclareMathOperator{\tr}{Tr}

\begin{document}

\title[Hardy inequalities for large fermionic systems --- \version]{Hardy inequalities for large fermionic systems}

\author[R. L. Frank]{Rupert L. Frank}
\address[Rupert L. Frank]{Mathe\-matisches Institut, Ludwig-Maximilians Universit\"at M\"unchen, The\-resienstr.~39, 80333 M\"un\-chen, Germany, and Munich Center for Quantum Science and Technology, Schel\-ling\-str.~4, 80799 M\"unchen, Germany, and Mathematics 253-37, Caltech, Pasa\-de\-na, CA 91125, USA}
\email{r.frank@lmu.de}

\author[T. Hoffmann-Ostenhof]{Thomas Hoffmann-Ostenhof}
\address[Thomas Hoffmann-Ostenhof]{University of Vienna, Department of Theoretical Chemistry, W\"ahringerstrasse 17, 1090 Wien, Austria}
\email{thoffmann@tbi.univie.ac.at}

\author[A. Laptev]{Ari Laptev}
\address[Ari Laptev]{Imperial College London, 180 Queen's Gate, London SW7 2AZ, UK, and Sirius Mathematics Center, Sirius University of Science and Technology, 1 Olympic Ave, 354340, Sochi, Russia}
\email{a.laptev@imperial.ac.uk}

\author[J. P. Solovej]{Jan Philip Solovej}
\address[Jan Philip Solovej]{Department of Mathematical Sciences, University of Copenhagen, Universitetsparken 5, 2100 Copenhagen, Denmark}
\email{solovej@math.ku.dk}

\thanks{\copyright\, 2024 by the authors. This paper may be reproduced, in its entirety, for non-commercial purposes.\\
Partial support through US National Science Foundation grant DMS-1954995 (R.L.F.), as well as through the Deutsche Forschungsgemeinschaft (DFG, German Research Foundation) through Germany’s Excellence Strategy EXC-2111-390814868 (R.L.F.) is acknowledged. \texttt{Add acknowledgements}}

\dedicatory{To Brian Davies, in admiration, on the occasion of his 80th birthday}

\begin{abstract}
	Given $0<s<\frac d2$ with $s\leq 1$, we are interested in the large $N$-behavior of the optimal constant $\kappa_N$ in the Hardy inequality $\sum_{n=1}^N (-\Delta_n)^s \geq \kappa_N \sum_{n<m} |X_n-X_m|^{-2s}$, when restricted to antisymmetric functions. We show that $N^{1-\frac{2s}d}\kappa_N$ has a positive, finite limit given by a certain variational problem, thereby generalizing a result of Lieb and Yau related to the Chandrasekhar theory of gravitational collapse.
\end{abstract}

\maketitle

\section{Introduction and main result}

A prototypical form of Hardy's inequality states that
$$
\int_{\R^d} |\nabla u(x)|^2\,dx \geq \frac{(d-2)^2}{4} \int_{\R^d} \frac{|u(x)|^2}{|x|^2}\,dx
$$
when $d\geq 3$ and $u\in\dot H^1(\R^d)$, the homogeneous Sobolev space. This and other forms of Hardy's inequality are fundamental tools in many questions in PDE, harmonic analysis, spectral theory and mathematical physics. We refer to the survey paper by Davies \cite{Da} and the books of Maz’ya \cite{Ma} and Opic and Kufner \cite{OpKu} for extensive results, as well as background and further references.

In \cite{HO2LaTi}, the second and third authors and their coauthors studied what they called \emph{many-particle Hardy inequalities}. These are inequalities for functions defined on $\R^{dN}$ with coordinates denoted by $X=(X_1,\ldots,X_N)$ with $X_1,\ldots, X_N \in\R^d$. Here $N\geq 2$ can be interpreted as the number of (quantum) particles in $\R^d$ and the $X_n$, $n=1,\ldots,N$, as their positions. The Hardy weight takes the form $\sum_{1\leq n<m\leq N} |X_n-X_m|^{-2}$. It is shown in \cite{HO2LaTi} that
\begin{equation}
	\label{eq:hardymany}
	\sum_{n=1}^N \int_{\R^{dN}} |\nabla_n u(X)|^2 \,dX \geq \beta_N^{(d)} \sum_{1\leq n<m\leq N} \int_{\R^{dN}} \frac{|u(X)|^2}{|X_n-X_m|^2}\,dX
\end{equation}
for all $u\in\dot H^1(\R^{dN})$ with a certain explicit lower bound on the optimal constant $\beta_d^{(N)}$ that is positive, provided $d\geq 3$ and $N\geq 2$. What is of interest is the behavior of the optimal constant $\beta_d^{(N)}$ as $N\to\infty$ for fixed $d$. This corresponds to the many-particle limit, a classical topic in mathematical physics. The explicit lower bound for $\beta_N^{(d)}$ obtained in \cite{HO2LaTi} shows that $\liminf_{N\to\infty} N^{-1} \beta_N^{(d)}>0$. It is also noted in \cite{HO2LaTi} that $\limsup_{N\to\infty} N^{-1} \beta_N^{(d)}<\infty$. The methods of this paper allow us to prove that $\lim_{N\to\infty} N^{-1} \beta_N^{(d)}$ exists and give an explicit expression for it in terms of a variational problem for functions on $\R^d$; see Subsection \ref{sec:bosons} for details.

Our main interest, however, lies in a variant of inequality \eqref{eq:hardymany}, namely its restriction to \emph{antisymmetric functions}. A function $u$ on $\R^{dN}$ is called antisymmetric if for any permutation $\sigma$ of $\{1,\ldots,N\}$ and a.e.\ $X\in\R^{dN}$ one has
$$
u(X_1,\ldots,X_N) = (\sgn\sigma) \, u(X_{\sigma(1)},\ldots,X_{\sigma(N)}) \,.
$$
Here $\sgn\sigma\in\{+1,-1\}$ denotes the sign of $\sigma$. The antisymmetry requirement appears naturally in physics in the description of fermions. (Note that we restrict ourselves here to scalar functions, corresponding to the spinless or spin-polarized situation, although in our results for dimension $d\geq 3$ spin could be incorporated; however, for $d=1,2$, when $|x|^{-2}$ is not locally integrable, the spin-polarization is crucial.)

We denote by $\kappa_N^{(d)}$ the optimal constant in inequality \eqref{eq:hardymany} when restricted to antisymmetric functions, that is,
$$
\kappa_N^{(d)} := \inf_{\substack{0\neq u\in\dot H^1(\R^{dN}) \\ \text{antisymmetric}}} \frac{\sum_{n=1}^N \int_{\R^{dN}} |\nabla_n u(X)|^2\,dX}{\sum_{1\leq n<m\leq N} \int_{\R^{dN}} \frac{|u(X)|^2}{|X_n-X_m|^2}\,dX} \,.
$$

As emphasized in \cite{HO2LaTi}, there are signifcant differences between the inequality on all functions in $\dot H^1(\R^{dN})$ and its restriction to antisymmetric ones. One important difference is that $\kappa_N^{(d)}>0$ for all $d\geq 1$ and $N\geq 2$, while $\beta_N^{(d)}=0$ for $d=1,2$ and $N\geq 2$; see \cite[Remarks 2.2(i) and Theorem 2.8]{HO2LaTi}. Remarkably, for $d=1$ the explicit value of the sharp constant $\kappa_N^{(1)}$ is known, $\kappa_N^{(1)}=\frac12$ for all $N\geq 2$ \cite[Theorem 2.5]{HO2LaTi}. For $d\geq 2$, as far as we know, only the lower bound $\kappa_N^{(d)}\geq d^2/N$ is known. This displays the same $N^{-1}$ behavior as $\beta_N^{(d)}$, but, as we will see in the present paper, this is not optimal, at least when $d\geq 3$.

In fact, our main result states that $\lim_{N\to\infty} N^{1-\frac2d} \kappa_N^{(d)}$ exists as a positive and finite number when $d\geq 3$, and gives an explicit expression for it in terms of a variational problem for functions on $\R^d$. 

This is the special case of a more general result, which concerns the inequality
\begin{equation}
	\label{eq:hardymanys}
	\sum_{n=1}^N \int_{\R^{dN}} |(-\Delta_n)^\frac s2 u(X)|^2 \,dX \geq \alpha_N^{(d,s)} \sum_{1\leq n<m\leq N} \int_{\R^{dN}} \frac{|u(X)|^2}{|X_n-X_m|^{2s}}\,dX
\end{equation}
for all $u\in\dot H^s(\R^{dN})$. Here $s$ is a real number satisfying $0<s<\frac d2$ and the operator $(-\Delta_n)^\frac s2$ acts on the $n$-th variable of $X=(X_1,\ldots,X_N)$ by multiplication by $|\xi_n|^s$ in Fourier space. The homogeneous Sobolev space $\dot H^s(\R^{dN})$ is the completion of $C^\infty_c(\R^{dN})$ with respect to the quadratic form on the left side of \eqref{eq:hardymanys}. It is relatively straightforward to see that
\begin{equation}
	\label{eq:tauspos}
	\inf_{N\geq 2} N \alpha_N^{(d,s)} > 0 \,.
\end{equation}
Indeed, for each pair $(n,m)$ with $n\neq m$ we have, by the ordinary fractional Hardy inequality (see Lemma \ref{hardydeltar} below),
\begin{equation}
	\label{eq:hardyfrac}
	\int_{\R^{d}} |(-\Delta_n)^\frac s2 u(X)|^2 \,dX_n \gtrsim \int_{\R^{d}} \frac{|u(X)|^2}{|X_n-X_m|^{2s}}\,dX_n \,.
\end{equation}
Integrating this inequality with respect to the remaining variables and summing over $n$ and $m$ gives \eqref{eq:tauspos}. As an aside, we mention that the optimal constant in \eqref{eq:hardyfrac} is known; see \cite{He,KoPeSe} and also \cite{Ya,FrSe}.

Our interest is again in the sharp constant in \eqref{eq:hardymanys} \emph{when restricted to antisymmetric functions}, that is, in
$$
\kappa_N^{(d,s)} := \inf_{\substack{0\neq u\in\dot H^s(\R^{dN}) \\ \text{antisymmetric}}} \frac{\sum_{n=1}^N \int_{\R^{dN}} |(\Delta_n)^\frac s2 u(X)|^2\,dX}{\sum_{1\leq n<m\leq N} \int_{\R^{dN}} \frac{|u(X)|^2}{|X_n-X_m|^{2s}}\,dX} \,.
$$

To state the limiting variational problem, we introduce for nonnegative, measurable functions $\rho$ on $\R^d$
$$
D_\lambda[\rho] := \frac12 \iint_{\R^d\times\R^d} \frac{\rho(x)\,\rho(x')}{|x-x'|^\lambda}\,dx\,dx' \,.
$$
Moreover, we denote 
$$
\tau_{d,s} := \inf_{0\leq\rho\in L^{1+\frac{2s}d}\cap L^1(\R^d)} \frac{\int_{\R^d} \rho(x)^{1+\frac{2s}d}\,dx \left( \int_{\R^d} \rho(x)\,dx \right)^{1-\frac{2s}d}}{D_{2s}[\rho]} \,.
$$
The fact that $\tau_d>0$ follows from the Hardy--Littlewood--Sobolev inequality \cite[Theorem 4.3]{LiLo}, together with H\"older's inequality. Finally, let
$$
c_{d,s}^{\rm TF} := \frac{(4\pi)^s}{1+\frac{2s}d} \, \Gamma(1+\tfrac d2)^\frac{2s}{d} \,.
$$
The superscript TF stands for `Thomas--Fermi' and it will become clear in the proof that this constant is related to the Thomas--Fermi approximation for the kinetic energy.

The following is our main result.

\begin{theorem}\label{main}
	Let $d\geq 1$ and $0<s<\frac d2$ with $s\leq 1$. Then
	\begin{equation}
		\label{eq:main}
			\lim_{N\to\infty} N^{1-\frac{2s}d} \, \kappa_N^{(d,s)} = \tau_{d,s} \, c_{d,s}^{\rm TF} \,.
	\end{equation}
\end{theorem}

\begin{remarks}
	(a) This result in the special case $s=\frac12$, $d=3$ is due to Lieb and Yau \cite{LiYa1}, following earlier work by Lieb and Thirring \cite{LiTh84}. While our overall strategy is similar to theirs, there are some significant differences, which we explain below.\\
	(b) Our proof of the asymptotics \eqref{eq:main} comes with remainder bounds. We show that $N^{1-\frac{2s}d} \, \kappa_N^{(d,s)}$ is equal to its limit up to a relative error of $\mathcal O(N^{-\frac{s(d-2s)}{d^2}})$; see \eqref{eq:lower} and \eqref{eq:upper}.\\
	(c) We believe that Theorem \ref{main} remains valid without the extra assumption $s\leq 1$. This would probably require significant additional effort at various places and, since our main interest is the case $s=1$, we decided to impose this simplifying assumption.\\
	(d) Theorem \ref{main} extends to the case where spin is taken into account, except that the limiting expression in \eqref{eq:main} is multiplied by a power of the number of spin states. We refer to \cite{LiSe} for an explanation of this terminology and to \cite{LiYa1} for proofs where spin is taken into account.\\
	(e) Finding the asymptotic behavior of $\kappa_N^{(d,s)}$ in the case $d=2s$ is an \emph{open problem}. In Appendix \ref{sec:upper2d} we discuss a conjecture of what might be the right order and prove the corresponding upper bound.
\end{remarks}

Let us give some background on Theorem \ref{main} and explain some aspects of its proof. The basic idea is that it is a combined semiclassical and mean-field limit. Such a limit is behind what is called Thomas--Fermi approximation for Coulomb systems and has first been made rigorous by Lieb and Simon in \cite{LiSe}. Parts of this proof were simplified through the use of coherent states \cite{Th,Li} and the Lieb--Thirring inequality \cite{LiTh,Li}, and these will also play an important role for us. For a recent study of this combined semiclassical and mean-field limit for quite general systems we refer to \cite{FoLeSo}.

One difficulty that we face here, compared to the analysis of nonrelativistic Coulomb systems \cite{LiSi} or the systems in \cite{FoLeSo}, is that the kinetic energy and the potential energy scale in the same way, so that there is no natural scale. This problem was overcome by Lieb and Yau \cite{LiYa1}, following earlier work of Lieb and Thirring \cite{LiTh84}, in their rigorous derivation of Chandrasekhar theory of gravitational collapse of stars. In important ingredient in the proofs of \cite{LiTh84,LiYa1} and also in the more recent \cite{FoLeSo}, is the \emph{L\'evy-Leblond method}. This method will also play a crucial role in our proof. It consists in dividing the $N$ particles into two groups, treating one part as `electrons' and the other part as `nuclei'. The electrons repel each other, and similarly the nuclei, while electrons and nuclei attract each other. The construction involves a further, free parameter that corresponds to the quotient between the charges of the electrons and the nuclei. At the end one averages over all such partitions.

There is one important structural property, however, that Lieb and Yau can take advantage of and we cannot. They deal with the case $s=\frac12$, $d=3$, where the interaction potential $|x|^{-1}$ is, up to a constant, the fundamental solution of the Laplacian and the corresponding (sub/super)harmonicity properties enter into the proof of \cite[Lemma 1]{LiYa1}. The same phenomenon occurs, for instance, for $s=1$, $d=4$, or in general for $s=\frac{d-2}{2}$, $d\geq 3$, but in the general case the interaction $|x|^{-2s}$ is not harmonic outside of the origin. Therefore a substantial part of our effort goes into proving bounds for systems interacting through Riesz potentials $|x|^{-2s}$ for general exponents $0<2s<d$; see Propositions \ref{electrostatic} and \ref{indirect}. In both cases our proof relies on the Fefferman--de la Llave decomposition of this interaction potential. As a curiosity, we mention that also the Cwikel--Lieb--Rozenblum inequality, and therefore the Lieb--Thirring inequality, which is another important ingredient in our proof, can be established using the Fefferman--de la Llave decomposition \cite{Fr2}. For more on this decomposition, see also \cite{HaSe}.

Neither the L\'evy-Leblond method nor the Fefferman--de la Llave decomposition seem to work for $d\leq 2s$ and this case remains \emph{open} (except for $s=d=1$). In Appendix \ref{sec:upper2d} we give a suggestion of what might be the relevant mechanism in the borderline case $d=2s$.

Finally we mention that the results in this paper, with the exception of those in Section \ref{sec:lower}, are contained in a preprint with the same title, dated October 30, 2006, that was circulated among colleagues. The present paper corrects some minor mistakes therein and adds a proof of the sharp asymptotic lower bound.

\medskip

It is our pleasure to dedicate this paper to Brian Davies in admiration of his many profound contributions to spectral theory and mathematical physics and, in particular, to the topic of Hardy inequalities. Happy birthday, Brian!


\section{Lower bound}\label{sec:lower}

Our goal in this section is to prove the lower bound in Theorem \ref{main}. That is, we will show
$$
\liminf_{N\to\infty} N^{1-\frac{2s}d} \kappa_N^{(d,s)} \geq \tau_{d,s} \, c_{d,s}^{\rm TF} \,.
$$
More precisely, we will prove the following quantitative version of it,
\begin{equation}
	\label{eq:lower}
	N^{1-\frac2d} \kappa_N \geq \tau_{d,s} \, c_{d,s}^{\rm TF} \left( 1 - \const N^{-\frac{s(d-2s)}{d^2}} \right).
\end{equation}
As explained in the introduction, we mostly follow the method in \cite{LiYa1}, but an important new ingredient, which replaces their \cite[Equation (2.21)]{LiYa1}, is the electrostatic inequality in Proposition \ref{electrostatic}.


\subsection{An electrostatic inequality}\label{sec:electrostatic}

For a probability measure $\mu$ on $\R^{dM}$ we denote by $\rho_\mu$ the nonnegative measure on $\R^d$ obtained by summing the marginals of $\mu$. That is, for any bounded continuous function $f$ on $\R^d$ we have
$$
\sum_{m=1}^M \int_{\R^{dN}} f(Y_m) \,d\mu(Y) = \int_{\R^d} f(y)\,d\rho_\mu(y) \,.
$$
The definition of $D_\lambda$ is extended to nonnegative measures on $\R^d$ in a natural way, namely, by
$$
D[\nu] := \frac12 \iint_{\R^d\times\R^d} \frac{d\nu(y)\,d\nu(y')}{|y-y'|^\lambda} \,.
$$

Next, for $R=(R_1,\ldots,R_K)\in\R^{dK}$ and $y\in\R^d$, we set
$$
\delta_R(y):= \min_{k=1,\ldots, K} |y-R_k| \,.
$$

\begin{proposition}\label{electrostatic}
	Let $d\geq 1$ and $0<\lambda<d$. Then for any $M,K\in\N$, $R=(R_1,\ldots,R_K)\in\R^{dK}$, $Z>0$ and any probability measure $\mu$ on $\R^{dM}$
	$$
	\int_{\R^{dM}} \sum_{m,k} \frac{Z}{|Y_m-R_k|^\lambda}\,d\mu(Y)  -  \sum_{k<l} \frac{Z^2}{|R_k-R_l|^\lambda} - D_\lambda[\rho_\mu] \lesssim Z \int_{\R^d} \frac{d\rho_\mu(y)}{\delta_R(y)^\lambda} \,,
	$$
	where the implied constant depends only on $d$ and $\lambda$.
\end{proposition}

The following proof uses some ideas from that of \cite[Corollary 1]{FedlL}.

\begin{proof}
	According to the Fefferman--de la Llave formula \cite{FedlL}, we have for all $y,y'\in\R^d$
	\begin{equation}
		\label{eq:fdll}
		\frac{1}{|y-y'|^\lambda} = \const \int_0^\infty \frac{dr}{r^{d+\lambda+1}}\, \int_{\R^d} da \, \1_{B_r(a)}(y) \, \1_{B_r(a)}(y')
	\end{equation}
	with a constant depending only on $d$ and $\lambda$. This implies that
	\begin{align*}
		& \int_{\R^{dM}} \sum_{m,k} \frac{Z}{|Y_m-R_k|^\lambda}\,d\mu(Y) -  \sum_{k<l} \frac{Z^2}{|R_k-R_l|^\lambda} - D_\lambda[\rho_\mu] \\
		& = \const \int_0^\infty \frac{dr}{r^{d+\lambda+1}}\, \int_{\R^d} da \left(
		\int_{\R^{dM}} Z \sum_{m,k} \1_{B_r(a)}(Y_m) \, \1_{B_r(a)}(R_k) \,d\mu(Y) \right.\\
		& \qquad\qquad\qquad\qquad\qquad\qquad - Z^2 \sum_{k<l} \1_{B_r(a)}(R_k) \, \1_{B_r(a)}(R_l) \\ 
		& \qquad\qquad\qquad\qquad\qquad\qquad \left. - \frac12 \iint_{\R^d\times\R^d} \rho_\mu(y) \1_{B_r(a)}(y) \, \1_{B_r(a)}(y') \rho_\mu(y')  \right) \\
		& = \const \int_0^\infty \frac{dr}{r^{d+\lambda+1}}\, \int_{\R^d} da \left( Z n_{B_r(a)} K_{B_r(a)} - \frac12 Z^2 K_{B_r(a)} ( K_{B_r(a)} - 1) - \frac12 n_{B_r(a)}^2 \right),
	\end{align*}
	where we have introduced, for any ball $B$,
	\begin{align*}
		n_B := \rho_\mu(B)
		\qquad\text{and}\qquad
		K_B := \sum_k \1_B(R_k) \,.
	\end{align*}
	Note that $K_B$ is a nonnegative integer. We claim that for any $n\geq 0$ and any $K' \in\N_0$,
	$$
	Z n K' - \frac12 Z^2 K'(K'-1) - \frac12 n^2 \leq Z n \, \1(K'\geq 1) \,.
	$$
	Indeed, this is true when $K'=0$, and when $K'\geq 1$, we write the left side as
	$$
	- \frac12 (n - Z \sqrt{K'(K'-1)})^2 + nZ \left( K'- \sqrt{K'(K'-1)} \right)
	$$
	and bound $K'- \sqrt{K'(K'-1)} \leq 1$. 
	
	Thus, we have shown that
	 \begin{align*}
	 	& \int_{\R^{dM}} \sum_{m,k} \frac{Z}{|Y_m-R_k|^\lambda}\,d\mu(Y) -  \sum_{k<l} \frac{Z^2}{|R_k-R_l|^\lambda} - D_\lambda[\rho_\mu] \\
	 	& \leq \const Z \int_0^\infty \frac{dr}{r^{d+\lambda+1}}\, \int_{\R^d} da \, n_{B_r(a)} \1(K_{B_r(a)}\geq 1) \\
	 	& = \const Z \int_{\R^d} d\rho_\mu(y)\, \int_0^\infty \frac{dr}{r^{d+\lambda+1}}\, \int_{\R^d} da\, \1_{B_r(a)}(y) \, \1(K_{B_r(a)}\geq 1) \,.
	 \end{align*}
	 Next, we bound
	 $$
	 \1_{B_r(a)}(y) \, \1(K_{B_r(a)}\geq 1) \leq \1_{B_r(a)}(y) \, \1(\delta_R(y)< 2r) \,.
	 $$
	 Indeed, when the left side does not vanish, we have $|y-a|<r$ and there is a $k\in\{1,\ldots,K\}$ such that $|R_k-a|<r$. Consequently $\delta_R(y) \leq |y-R_k|< 2r$. 
	 
	 By performing first the $a$ and then the $r$ integration, we obtain for each $y\in\R^d$,
	 \begin{align*}
	 	\int_0^\infty \frac{dr}{r^{d+\lambda+1}} \int_{\R^d} da \, \1_{B_r(a)}(y) \, \1(K_{B_r(a)}\geq 1)
	 	& \leq \int_0^\infty \frac{dr}{r^{d+\lambda+1}} \int_{\R^d} da \, \1_{B_r(a)}(y) \, \1(\delta_R(y)< 2r) \\
	 	& = \const \int_0^\infty \frac{dr}{r^{\lambda+1}}\, \1(\delta_R(y)< 2r) \\
	 	&= \const \frac{1}{\delta_R(y)^\lambda} \,.
	 \end{align*}
	 This implies the claimed inequality.	 
\end{proof}


\subsection{Lieb--Thirring inequality}

Associated to a normalized function $\psi\in L^2(\R^{dM})$ is a probability measure $d\mu(Y)=|\psi(Y)|^2\,dY$ and therefore we can consider the measure $d\rho_\mu$ on $\R^d$ as in the previous subsection. In the present case, this measure turns out to be absolutely continuous and we denote its density by $\rho_\psi$. Explicitly,
$$
\rho_\psi(y) := \sum_{m=1}^M \int_{\R^{3(M-1)}} |\psi(\ldots,Y_{m-1},y,Y_{m+1},\ldots)|^2 \,dY_1 \ldots dY_{n-1}\,dY_{n+1}\ldots dY_M \,.
$$
This density appears in the following famous Lieb--Thirring inequality.

\begin{lemma}\label{lt}
	Let $d\geq 1$ and $s>0$. Then for any $M\in\N$ and any antisymmetric and $L^2$-normalized $\psi\in H^s(\R^{dM})$,
	$$
	\sum_{m=1}^M \int_{\R^{dM}} |(-\Delta_m)^\frac s2\psi|^2\,dY \gtrsim \int_{\R^d} \rho_\psi^{1+\frac{2s}d}\,dy \,,
	$$
	where the implied constant depends only on $d$ and $s$.
\end{lemma}

For $s=1$, this inequality is due to Lieb and Thirring \cite{LiTh}. Their original proof generalizes readily to the full regime $s>0$; see also \cite[Chapter 4]{LiSe}, as well as \cite[Theorem 4.60 and Section 7.4]{FrLaWe} and \cite{Fr}. For the currently best known values of the constants, see \cite{FrHuJeNa}.


\subsection{Coherent states}

The following lemma is a rigorous version of the Thomas--Fermi approximation for the kinetic energy. It is proved with the help of coherent states.

\begin{lemma}\label{coherent}
	Let $d\geq 1$ and $0< s\leq 1$. Let $g\in H^s(\R^d)$ be $L^2$-normalized and, when $s>\frac12$, assume that $|\widehat g|$ is even under $\xi\mapsto-\xi$. Then for any antisymmetric, $L^2$-normalized $\psi\in H^s(\R^{dM})$,
	$$
	\sum_{m=1}^M \int_{\R^{dM}} |(-\Delta_m)^\frac s2 \psi|^2\,dY \geq c^{\rm TF}_{d,s} \int_{\R^d} (\rho_\psi * |g|^2)^{1+\frac{2s}d}\,dy - M \|(-\Delta)^\frac s2 g\|^2 \,.
	$$
\end{lemma}

Here $\widehat g(\xi) = (2\pi)^{-\frac d2} \int_{\R^d} e^{-i\xi\cdot x} g(x)\,dx$ denotes the Fourier transform of $g$.

Bounds of the same type as in the lemma appear in \cite[Eqs.\ (5.14)--(5.22)]{Li} in the special case $s=1$ and $d=3$; a general version is formulated in \cite[Lemma B.4]{LiYa1}. Because of a subtlety in the application of that lemma, we sketch the proof.

\begin{proof}
	We will show that for any $L^2$-normalized $v\in H^s(\R^d)$, one has
	\begin{equation}
		\label{eq:coherentbound}
		\|(-\Delta)^\frac s2 v \|^2 \geq \iint_{\R^d\times\R^d} |\eta|^{2s} |\widetilde v(\eta,y) |^2 \, \frac{d\eta\,dy}{(2\pi)^d} - \|(-\Delta)^\frac s2 g\|^2 \,,
	\end{equation}
	where
	$$
	\widetilde v(\eta,y) := \int_{\R^d} e^{-i\eta\cdot x} \overline{g(y-x)} v(x)\,dx \,.
	$$
	(Compared to \cite[Section 12.7]{LiLo} and other presentations, we find it convenient to use $y-x$ instead of $x-y$ in the definition of $\widetilde v$.) Once \eqref{eq:coherentbound} is shown, the inequality in the lemma follows as in \cite[Lemma B.3]{LiYa1}.
	
	To prove \eqref{eq:coherentbound}, we observe that
	$$
	\int_{\R^d} |\widetilde v(\eta,y) |^2 \,dy = (2\pi)^d \int_{\R^d} |\widehat v(\xi)|^2 |\widehat g(\eta-\xi)|^2\,d\xi \,.
	$$
	We multiply this identity by $|\eta|^{2s}$ and integrate with respect to $\eta$. 
	
	In case $s\leq\frac12$ we use the subadditivity of $t\mapsto t^{2s}$ to bound on the right side
	$$
	|\eta|^{2s} \leq |\xi|^{2s} + |\eta - \xi|^{2s}
	$$
	and obtain the claimed inequality \eqref{eq:coherentbound}. (This is essentially the argument in \cite[Lemma B.3]{LiYa1}.) 
	
	In case $\frac12<s\leq 1$ we use the evenness of $|\widehat g|$ to replace $|\eta|^{2s}= |\xi+(\eta-\xi)|^{2s}$ by $\frac12(|\xi+(\eta-\xi)|^{2s}+ |\xi-(\eta-\xi)|^{2s}$. We then apply the elementary inequality
	\begin{equation}
		\label{eq:elementary}
		\frac12 ( |\xi+\zeta|^{2s} +|\xi-\zeta|^{2s}) \leq (|\xi|^2 + |\zeta|^2)^s \leq |\xi|^{2s} + |\zeta|^{2s} \,.
	\end{equation}
	and argue similarly as for $s\leq\frac12$ to obtain the claimed inequality \eqref{eq:coherentbound}. The second inequality in \eqref{eq:elementary} follows from the subadditivity of $t\mapsto t^s$, and to prove the first inequality we write
	$$
	\frac12 \left( |\xi+\zeta|^{2s} +|\xi-\zeta|^{2s} \right) = \frac12 \left( (1+t_*)^s + (1-t_*)^s \right)  \left(|\xi|^2 + |\zeta|^2 \right)^s
	$$
	with $t_* = 2\xi\cdot\zeta/(|\xi|^2+|\zeta|^2)\in[-1,1]$ and note that $[-1,1]\ni t\mapsto (1+t)^s + (1-t)^s$ attains its maximum at $t=0$.
\end{proof}

We expect a similar bound as in Lemma \ref{coherent} told for $s>1$ as well, but the structure of the remainder term will probably more complicated.


\subsection{Summary so far}

Let us combine the bounds from this section.

\begin{corollary}\label{summary}
	Let $d\geq 1$, $0<s<\frac d2$ with $s\leq 1$, $M\in\N$, $K\geq 2$, $R=(R_1,\ldots,R_K)\in\R^{dK}$ and $Z>0$. Then for any antisymmetric and $L^2$-normalized $\psi\in H^s(\R^{dM})$,
	\begin{align*}
		& \left\langle \psi, \sum_{m,k} \frac{Z}{|Y_m-R_k|^{2s}} \psi \right\rangle -  \sum_{k<l} \frac{Z^2}{|R_k-R_l|^{2s}} \\
		& \leq \left( 1 + \const M^{-\frac{s(d-2s)}{d^2}} \right) (\tau_{d,s} \, c_{d,s}^{\rm TF})^{-1} M^{1-\frac {2s}d} \sum_{m=1}^M \int_{\R^{dM}} |(-\Delta_m)^\frac s2 \psi|^2\,dY \\
		& \quad 
		+ \const Z \int_{\R^d} \frac{\rho_\psi(y)}{\delta_R(y)^{2s}}\,dy \,.
	\end{align*}
\end{corollary}

\begin{proof}
	Let $\psi\in H^s(\R^{dM})$ be antisymmetric and $L^2$-normalized. We recall that according to Proposition \ref{electrostatic} we have
	\begin{align}\label{eq:summaryproof}
		\left\langle \psi, \sum_{m,k} \frac{Z}{|Y_m-R_k|^{2s}} \psi \right\rangle -  \sum_{k<l} \frac{Z^2}{|R_k-R_l|^{2s}}
		\leq D_{2s}[\rho_\psi] + \const Z \int_{\R^d} \frac{\rho_\psi(y)}{\delta_R(y)^{2s}}\,dy \,.
	\end{align}	
	The second term on the right side appears in the claimed error bound. To bound the first term, let $g\in H^s(\R^d)$ be $L^2$-normalized. Using the definition of $\tau_{d,s}$, Young's convolution inequality and Lemma \ref{coherent} we find
	\begin{align*}
		D_{2s}[\rho_\psi]
		& \leq \tau_{d,s}^{-1} \int_{\R^d} (\rho_\psi*|g|^2)^{1+\frac{2s}d}\,dy \left( \int_{\R^d} \rho_\psi * |g|^2 \,dy \right)^{1-\frac{2s}d} + \left( D_{2s}[\rho_\psi] - D_{2s}[\rho_\psi * |g|^2] \right) \\
		& \leq \tau_{d,s}^{-1} M^{1-\frac {2s}d} \int_{\R^d} (\rho_\psi*|g|^2)^{1+\frac{2s}d}\,dy
		+ \frac12 \| \rho_\psi \|_{1+\frac{2s}d}^2 \left\| |x|^{-2s} \!-\! |g|^2 * |x|^{-2s} * |g|^2 \right\|_{\frac{d+2s}{4s}} \\
		& \leq (\tau_{d,s} \, c_{d,s}^{\rm TF})^{-1}  M^{1-\frac {2s}d} \sum_{m=1}^M \int_{\R^{dM}} |(-\Delta_m)^\frac s2 \psi|^2\,dY + \mathcal R
	\end{align*}
	with
	$$
	\mathcal R:= (\tau_{d,s} \, c_{d,s}^{\rm TF})^{-1} M^{2-\frac {2s}d} \|(-\Delta)^\frac s2 g\|^2 +  \frac12 \| \rho_\psi \|_{1+\frac{2s}d}^2 \left\| |x|^{-2s} - |g|^2 * |x|^{-2s} * |g|^2 \right\|_{\frac{d+2s}{4s}} \,.
	$$
	We now assume that $g(x) = \ell^{-\frac d2} G( \ell^{-1} x)$ for an $L^2$-normalized function $G\in H^s(\R^d)$ and a parameter $\ell>0$ to be chosen. We have
	\begin{align*}
		\mathcal R & = \ell^{-2s} (\tau_{d,s} \, c_{d,s}^{\rm TF})^{-1} M^{2-\frac{2s}d} \|(-\Delta)^\frac s2 G\|^2 \\
		& \quad +  \ell^{\frac{2s(d-2s)}{d+2s}} \frac12 \| \rho_\psi \|_{1+\frac{2s}d}^2 \left\| |x|^{-2s} - |G|^2 * |x|^{-2s} * |G|^2 \right\|_{\frac{d+2s}{4s}} \,.
	\end{align*}
	We note that the function $|x|^{-2s} - |G|^2 * |x|^{-2s} * |G|^2$ behaves like $|x|^{-2s}$ as $|x|\to 0$. Moreover, assuming that $|G|$ is even and that $|x|^2 |G|^2$ is integrable, it is easy to see that $|x|^{-2s} - |G|^2 * |x|^{-2s} * |G|^2=O(|x|^{-2s-2})$ as $|x|\to\infty$. A tedious, but elementary analysis shows that $(2s+2)\frac{d+2s}{4s}>d$. (Indeed, this is equivalent to $2s^2 - s(d-2)+d>0$. This is always satisfied when $d=1,2$. For $d\geq 3$ the left side is decreasing with respect to $s$ for $s\leq\frac{d-2}4$, so its validity for our parameter values follows from its validity at $s=\min\{1,\frac{d-2}4\}$.) The conclusion of this discussion is that $|x|^{-2s} - |G|^2 * |x|^{-2s} * |G|^2\in L^{\frac{d+2}{4}}(\R^d)$. We consider $G$ as fixed and choose
	$$
	\ell = M^{\frac{(d-s)(d+2s)}{2sd^2}} \|\rho_\psi\|_{1+\frac{2s}d}^{-\frac{d+2s}{2sd}}
	$$
	in order to balance the two error terms, and obtain
	$$
	\mathcal R \lesssim M^{-\frac{s(d-2s)}{d^2}} M^{1-\frac{2s}d} \|\rho_\psi\|_{1+\frac{2s}d}^{1+\frac{2s}d} \,.
	$$
	Using the Lieb--Thirring inequality (Proposition \ref{lt}) we can further bound the right side and arrive at
	$$
	\mathcal R \lesssim M^{-\frac{s(d-2s)}{d^2}} M^{1-\frac{2s}d} \sum_{m=1}^M \int_{\R^{dM}} |(-\Delta_m)^\frac s2 \psi|^2\,dY \,.
	$$
	This implies the claimed bound
\end{proof}


\subsection{Domination of the nearest neighbor attraction}

For $X=(X_1,\ldots,X_N)$ and $n\in\{1,\ldots,N\}$, let
$$
\delta_n(X) := \min_{m\neq n} |X_n-X_m| \,.
$$

\begin{proposition}\label{nearestneighbor}
	Let $d\geq 1$ and $0<s<\frac d2$ with $s\leq 1$. Then for any antisymmetric $u\in\dot H^s(\R^{dN})$ we have
	$$
	\sum_{n=1}^N \int_{\R^{dN}} |(-\Delta_n)^\frac s2 u(X)|^2\,dX \gtrsim \sum_{n=1}^N \int_{\R^{dN}} \frac{|u(X)|^2}{\delta_n(X)^{2s}}\,dX
	$$
	with an implicit constant that only depends on $d$ and $s$.
\end{proposition}

This bound appears as \cite[Theorem 5]{LiYa2} in the cases $d=3$ and $s\in\{\frac12,1\}$, but the proof readily generalizes to the stated parameter regime and is omitted. We also mention an alternative proof in \cite[Corollary 2]{FedlL}, which is based on a Fefferman--de la Llave type formula for the $\dot H^s(\R^d)$-seminorm and which generalizes to the regime $s<1$.

Probably Proposition \ref{nearestneighbor} remains valid for $1<s<\frac d2$, but this would require an argument and for the sake of brevity we do not consider this case. The IMS localization formula in \cite[Lemma 14]{LuNaPo} might be useful.


\subsection{Proof of the lower bound in Theorem \ref{main}}

We turn to the proof of \eqref{eq:lower}, for which we use the Levy-Leblond method \cite{LL}, similarly as in \cite{LiTh84,LiYa1}. Given $N\geq 3$ we choose an integer $M\in\{1,\ldots,N-2\}$ and a real number $Z>0$. We set $K:=N-M$ and consider partitions $\pi=(\pi_1,\pi_2)$ of $\{1,\ldots,N\}$ into two disjoint sets $\pi_1$ and $\pi_2$ with $M$ and $K$ elements, respectively. We have
\begin{align}\label{eq:levyleblond}
	\sum_{n<m} \frac{1}{|X_n-X_m|^{2s}} & = \frac{M(N-1)}{2ZMK - Z^2K(K-1)} \frac{N}{M} \binom{N}{M}^{-1} \notag \\
	& \quad \times \sum_\pi \left( \sum_{m\in\pi_1} \sum_{k\in\pi_2} \frac{Z}{|X_m-X_k|^{2s}} - \sum_{k<l\in\pi_2} \frac{Z^2}{|X_k-X_l|^{2s}} \right).
\end{align}

Let $u\in \dot H^s(\R^{dN})$ be antisymmetric. Our goal is to bound the Hardy quotient for $u$. By density we may assume that $u\in H^s(\R^{dN})$ and by homogeneity we may assume that $u$ is $L^2$-normalized. We integrate the left side of \eqref{eq:levyleblond} against $|u(X)|^2$. Correspondingly, on the right side we obtain a sum over partitions and we bound the integral for each fixed such partition $P$. We first carry out the integral over the variables in $\pi_1$. Denoting these variables as $(Y_1,\ldots,Y_M)$ and the variables in $\pi_2$ as $(R_1,\ldots,R_K)$ we infer from Corollary \ref{summary} that
\begin{align*}
	& \int_{\R^{dM}} \left( \sum_{m\in\pi_1} \sum_{k\in\pi_2} \frac{Z}{|X_m-X_k|^{2s}} - \sum_{k<l\in\pi_2} \frac{Z^2}{|X_k-X_l|^{2s}} \right) |u(X_1,\ldots,X_N)|^2 \,d\pi_1(X) \\
	& \leq \left( 1 + \const M^{-\frac{s(d-2s)}{d^2}} \right) (\tau_{d,s} \, c_{d,s}^{\rm TF})^{-1} M^{1-\frac{2s}d} \sum_{m\in\pi_1} \int_{\R^{dM}} |(\Delta_m)^\frac s2 u(X_1,\ldots,X_N)|^2\,d\pi_1(X) \\
	& \quad + \const Z \sum_{m\in\pi_1} \int_{\R^{dM}} \frac{|u(X_1,\ldots,X_N)|^2}{\delta_{\pi_2(X)}(X_m)^{2s}}\,d\pi_1(X) \,.
\end{align*}
Here $d\pi_1(X)$ denotes integration with respect to the variables $X_m$ with $m\in\pi_1$ and, for $m\in\pi_1$, we have 
$$
\delta_{\pi_2(X)}(X_m) = \min_{k\in\pi_2} |X_m-X_k| \geq \min_{k\neq m} |X_m-X_k| = \delta_m(X) \,.
$$
Inserting this into the above bound and carrying out the integration over the variables in $\pi_2$, we obtain
\begin{align*}
	& \int_{\R^{dN}} \left( \sum_{m\in\pi_1} \sum_{k\in\pi_2} \frac{Z}{|X_m-X_k|^{2s}} - \sum_{k<l\in\pi_2} \frac{Z^2}{|X_k-X_l|^{2s}} \right) |u(X)|^2 \,dX \\
	& \leq \left( 1 + \const M^{-\frac{s(d-2s)}{d^2}} \right) (\tau_{d,s} \, c_{d,s}^{\rm TF})^{-1} M^{1-\frac{2s}d} \sum_{m\in\pi_1} \int_{\R^{dN}} |(-\Delta_m)^\frac s2 u(X)|^2\,dX \\
	& \quad + \const Z \sum_{m\in\pi_1} \int_{\R^{dN}} \frac{|u(X_1,\ldots,X_N)|^2}{\delta_m(X)^{2s}}\,dX \,.
\end{align*}
According to \eqref{eq:levyleblond}, summing this bound over $\pi$ gives
\begin{align*}
	& \int_{\R^{dN}} \frac{|u(X)|^2}{|X_n-X_m|^{2s}}\,dX \\
	& \leq \frac{M(N-1)}{2ZMK - Z^2K(K-1)} \left( 1 + \const M^{-\frac{s(d-2s)}{d^2}} \right) \\
	& \quad \quad \times (\tau_{d,s} \, c_{d,s}^{\rm TF})^{-1} M^{1-\frac{2s}d} \sum_{n=1}^N \int_{\R^{dN}} |(-\Delta_n)^\frac s2 u|^2\,dX \\
	& \quad + \const \frac{M(N-1)}{2ZMK - Z^2K(K-1)} Z \sum_{n=1}^N \int_{\R^{dN}} \frac{|u(X)|^2}{\delta_n(X)^{2s}}\,dX \,.
\end{align*}
Using Proposition \ref{nearestneighbor}, the right side can be bounded by
$$
C \, (\tau_{d,s} \, c_{d,s}^{\rm TF})^{-1} N^{1-\frac{2s}d} \sum_{n=1}^N \int_{\R^{dN}} |(-\Delta_n)^\frac s2 u|^2\,dX \,.
$$
with
$$
C := \frac{M(N-1)}{2ZMK - Z^2K(K-1)} \left( 1 + \const M^{-\frac{s(d-2s)}{d^2}} + \const Z M^{-1+\frac{2s}d} \right) \left( \frac{M}{N} \right)^{1-\frac{2s}d}.
$$

Our goal is to choose the parameters $M$ and $Z$ (depending on $N$) in such a way that $C\to 1$ as $N\to\infty$. We choose $Z=M/K$ and obtain
$$
C = \frac{1+M^{-1}(K-1)}{1+K^{-1}} \left( 1 + \const M^{-\frac{s(d-2s)}{d^2}} + \const K^{-1} M^{\frac{2s}d} \right) \left( \frac{M}{N} \right)^{1-\frac{2s}d}.
$$
With the choice
$$
K := [ N ^{\frac sd+\frac12} ]
$$
we find
$$
C \leq 1 + \const N^{-\frac{s(d-2s)}{d^2}} \,.
$$
This completes the proof of \eqref{eq:lower}.
\qed

\begin{remark}
	Under the additional assumption $d>4s$ one can prove \eqref{eq:lower} (with a worse remainder bound) without using Proposition \ref{nearestneighbor}. Indeed, inserting the bound from Lemma \ref{hardydeltar} below into the bound in Corollary \ref{summary}, we can drop the last term there at the expense of replacing the factor in front of the first term by
	$$
	1 + \const M^{-\frac{s(d-2s)}{d^2}} + \const Z K^{\frac{2s}d} M^{-1+\frac {2s}d} \,.
	$$
	Choosing again $Z=M/K$, we can choose $K \sim N^\frac{d+2s}{2(d-s)}$ and arrive at \eqref{eq:lower} with the remainder $1- \const N^{-\frac{d-4s}{2(d-s)}}$.
\end{remark}

\begin{lemma}\label{hardydeltar}
	Let $0<s<\frac d2$. Then for all $v\in\dot H^s(\R^d)$, $K\in\N$ and $R\in\R^{3K}$
	$$
	\| (-\Delta)^\frac s2 v \|_2^2 \gtrsim K^{-\frac{2s}d} \int_{\R^d} \frac{|v(x)|^2}{\delta_R(x)^{2s}}\,dx
	$$
	with an implicit constant depending only on $d$ and $s$.
\end{lemma}

The following proof has some similarities with \cite[Lemma B.1]{LiYa1}.

\begin{proof}
	We use the improved Sobolev embedding in Lorentz spaces \cite{Pe} (see also \cite[Theorem 17.49]{Le}),
	$$
	\| (-\Delta)^\frac s2 v \|_2^2 \gtrsim \|v\|_{L^{\frac{2d}{d-2s},2}(\R^d)} \,,
	$$
	together with H\"older's inequality in Lorentz spaces \cite[Exercise 15.22]{Le},
	$$
	\int_{\R^d} \frac{|v(x)|^2}{\delta_R(x)^{2s}}\,dx \lesssim \|\delta_R^{-2s} \|_{L^{\frac{d}{2s},\infty}(\R^d)}  \| |v|^2 \|_{L^{\frac{d}{d-2s},1}(\R^d)} = \|\delta_R^{-2s} \|_{L^{\frac{d}{2s},\infty}(\R^d)}  \| v \|_{L^{\frac{2d}{d-2s},2}(\R^d)}^2 \,.
	$$
	It remains to bound the weak $L^{\frac d{2s}}$ norm of $\delta_R^{-2s}$. We have 
	\begin{align*}
		|\{ \delta_R < \lambda \}| \leq \sum_{k=1}^K |\{ |\cdot - R_k|<\lambda\}| = \const K \lambda^d \,,
	\end{align*}
	so
	$$
	\|\delta_R^{-2s} \|_{L^{\frac{d}{2s},\infty}(\R^d)} = \sup_{\mu>0} \mu |\{ \delta_R^{-2s} >\mu \}|^\frac{2s}{d} \lesssim K^\frac{2s}d \,.
	$$
	This gives the claimed bound.
\end{proof}


\subsection{The case without antisymmetry}\label{sec:bosons}

In this subsection we explain how the proof of \eqref{eq:upper} can be modified to give a lower bound on the optimal constant $\beta_N^{(d,s)}$ in \eqref{eq:hardymanys}. We denote
$$
\omega_{d,s} := \inf_{0\neq\sqrt{\rho}\in H^s(\R^d)} \frac{\|(-\Delta)^\frac s2 \sqrt\rho\|_2^2 \, \|\rho\|_1}{D_{2s}[\rho]} \,.
$$
It is not difficult to show that $\omega_{d,s}>0$ when $0<s<\frac d2$ and that there is an optimizer $\rho_*$; see, e.g., \cite[Theorem 4]{LiYa1} in the case $s=\frac12$, $d=3$. Then one can show that
\begin{equation}
	\label{eq:upperbosons}
	\beta_N^{(d,s)} \leq \frac{1}{N-1}\, \omega_{d,s}
	\qquad\text{for all}\ N\geq 2 \,.
\end{equation}
by taking $u(X)=\prod_{n=1}^N \sqrt{\rho_*(X_n)}$. For $s=1$, $d\geq 3$ this argument appears in \cite[Theorem 2.3]{HO2LaTi}. We now state the lower bound corresponding to \eqref{eq:upperbosons}.

\begin{proposition}
	Let $d\geq 1$ and $0<s<\frac d2$ with $s\leq 1$. Then
	$$
	N \beta_N^{(d,s)} \geq \omega_{d,s} \left( 1 - \const N^{-1+\frac{2s}d} \right).
	$$
\end{proposition}

\begin{proof}
	We proceed from inequality \eqref{eq:summaryproof}, which did not use the antisymmetry of $\psi$. Using the definition of $\omega_{d,s}$ and Lemma \ref{hardydeltar} we can bound the right side of \eqref{eq:summaryproof} by
	\begin{align*}
		& D_{2s}[\rho_\psi] + \const Z \int_{\R^d} \frac{\rho_\psi(y)}{\delta_R(y)^{2s}}\,dy \\
		& \leq \omega_{d,s}^{-1} \|(-\Delta)^\frac s2 \sqrt{\rho_\psi} \|_2^2 \, \|\rho_\psi \|_1 
		+ \const Z K^\frac{2s}{d} \|(-\Delta)^\frac s2 \sqrt{\rho_\psi} \|_2^2 \\
		& \leq \left( 1 + \const Z M^{-1} K^\frac{2s}d \right) \omega_{d,s}^{-1} \, M \sum_{m=1}^M \int_{\R^{dM}} |(-\Delta_m)^\frac s2 \psi|^2\,dY \,.
	\end{align*}
	The second inequality here is the Hoffmann-Ostenhof inequality \cite{HO2} for $s=1$ and its generalization to $s<1$ by Conlon \cite{Con}. Following the L\'evy-Leblond method we deduce from this bound that
	\begin{align*}
		\int_{\R^{dN}} \frac{|u(X)|^2}{|X_n-X_m|^{2s}}\,dX \leq C \omega_{d,s}^{-1} N \sum_{n=1}^N \int_{\R^{dN}} |(-\Delta_n)^\frac s2 \psi|^2\,dX
	\end{align*}
	with
	$$
	C:= \frac{M(N-1)}{2ZMK-Z^2K(K-1)} \left( 1 + \const Z M^{-1} K^\frac{2s}d \right) \frac{M}{N} \,.
	$$
	We choose again $Z=M/K$ and obtain
	$$
	C:= \frac{1 + M^{-1}(K-1)}{1 + K^{-1}} \left( 1 + \const K^{-1+\frac{2s}d} \right) \frac{M}{N} \,.
	$$
	Choosing $M=1$ we arrive at the claimed bound.
\end{proof}


\section{Upper bound}

Our goal in this section is to prove the upper bound in Theorem \ref{main}. That is, we shall show
$$
\limsup_{N\to\infty} N^{1-\frac{2s}d} \kappa_N^{(d,s)} \leq \tau_{d,s} \, c_{d,s}^{\rm TF} \,.
$$
More precisely, we will prove the following quantitative version of it,
\begin{equation}
	\label{eq:upper}
	N^{1-\frac{2s}d} \kappa_N^{(d,s)} \leq \tau_{d,s} \, c_{d,s}^{\rm TF} \left( 1 + \const N^{-\frac{s(d-2s)}{d^2}} \right).
\end{equation}
For the proof we follow rather closely the method in \cite[Section 3]{LiTh84}. One new ingredient is an exchange inequality, which appears in Proposition \ref{indirect}.


\subsection{A bound on the indirect part of the Riesz energy}

Here we return to the setting of Subsection \ref{sec:electrostatic} and consider probability measures $\mu$ on $\R^{dN}$ and their marginals $\rho_\mu$.

\begin{proposition}\label{indirect}
	Let $d\geq 1$ and $0<\lambda<d$. Then for any $N\in\N$ and for any nonnegative Borel probability measure $\mu$ on $\R^{dN}$ with $\rho_\mu\in L^{1+\frac\lambda d}(\R^d)$,
	$$
	\sum_{1\leq n<m\leq N} \int_{\R^{dN}} \frac{d\mu(X)}{|X_n-X_m|^\lambda} - D_\lambda[\rho_\mu] \gtrsim - \int_{\R^d} \rho_\mu(x)^{1+\frac \lambda d}\,dx
	$$
	with an implicit constant depending only on $d$ and $\lambda$.
\end{proposition}

This bound for $\lambda=1$ and $d=3$ is due to \cite{Li0} with an improved constant in \cite{LiOx}. Here we adapt the proof strategy from \cite{LiSoYn}, which does not use (sub/super)harmonicity properties of the interaction potential.

\begin{proof}
	The Fefferman--de la Llave formula \eqref{eq:fdll} implies that
	\begin{align*}
		& \sum_{1\leq n<m\leq N} \int_{\R^{dN}} \frac{d\mu(X)}{|X_n-X_m|^\lambda} - D_\lambda[\rho_\mu] \\
		& = \const \int_0^\infty \frac{dr}{r^{d+\lambda+1}}\, \int_{\R^d} da \left( \int_{\R^{dN}} \sum_{1\leq n<m\leq N} \1_{B_r(a)}(X_n) \, \1_{B_r(a)}(X_m) \,d\mu(X) \right. \\ 
		& \qquad\qquad\qquad\qquad\qquad\qquad \left. - \frac12 \iint_{\R^d\times\R^d} \rho_\psi(y) \1_{B_r(a)}(y) \, \1_{B_r(a)}(y') \rho_\psi(y')  \right) \\
		& = \const \int_0^\infty \frac{dr}{r^{d+\lambda+1}}\, \int_{\R^d} da \left( \mathcal I_{B_r(a)} - \frac12 n_{B_r(a)}^2 \right),
	\end{align*}
	where, for any ball $B$, $n_B$ is defined as in the proof of Proposition \ref{electrostatic} and where
	$$
	\mathcal I_B := \int_{\R^{dN}} \sum_{1\leq n<m\leq N} \1_{B}(X_n) \, \1_{B}(X_m) \,d\mu(X) \,.
	$$
	
	We will derive two different lower bounds on the integrand $\mathcal I_B - \frac12 n_B^2$. The first one is the trivial bound
	\begin{equation}
		\label{eq:lobound1}
		\mathcal I_B - \frac12 n_B^2 \geq - \frac12 n_B^2 \,.
	\end{equation}
	To derive the second lower bound, we estimate, with an arbitrary $0\leq\rho\in L^1_\loc(\R^d)$ and writing $\rho_B:= \int_B \rho(x)\,dx$,
	\begin{align*}
		\sum_{n<m} \1_{B}(X_n) \, \1_{B}(X_m) & = \rho_B \sum_n \1_B(X_n) - \frac12 \rho_B^2 - \frac12 \sum_n \1_B(X_n) \\
		& \quad + \frac12 \left( \sum_n \1_B(X_n) - \rho_B \right)^2 \\
		& \geq \rho_B \sum_n \1_B(X_n) - \frac12 \rho_B^2 - \frac12 \sum_n \1_B(X_n) \,.
	\end{align*}
	Thus,
	$$
	\mathcal I_B \geq \rho_B n_B -  \frac12 \rho_B^2 - \frac12 n_B \,,
	$$
	with $n_B$ as in the proof of Proposition \ref{electrostatic}. Choosing $\rho=\rho_\mu$ we have $\rho_B=n_B$ and therefore
	\begin{equation}
		\label{eq:lobound2}
		\mathcal I_B - \frac12 n_B^2 \geq - \frac12 n_B \,.
	\end{equation}
	
	Combining \eqref{eq:lobound1} and \eqref{eq:lobound2} we find
	$$
	\mathcal I_B - \frac12 n_B^2 \geq - \frac12 \min\{ n_B, n_B^2 \} \,.
	$$
	We now bound $n_B$ from above using the maximal function $\rho_\mu^*$ of $\rho_\mu$. By its definition (see, e.g., \cite[Definition 2.1.1]{Gr}) we obtain
	$$
	n_{B_r(a)} \leq |B_r(a)| \, \rho_\mu^*(a) = \const r^d\, \rho_\mu^*(a) \,.
	$$
	Consequently,
	\begin{align*}
		\int_0^\infty \frac{dr}{r^{d+\lambda+1}}\, \int_{\R^d} da \left( \mathcal I_{B_r(a)} - \frac12 n_{B_r(a)}^2 \right) 
		& \lesssim \int_0^\infty \frac{dr}{r^{d+\lambda+1}}\, \int_{\R^d} da\, \min\{ r^d \rho_\mu^*(a), r^{2d} \rho_\mu^*(a)^2\} \\
		& = \const \int_{\R^d} da\,  \rho_\mu^*(a)^{1+\frac\lambda d} \,.
	\end{align*}
	The assertion now follows from the boundedness of the maximal function on $L^{1+\frac{\lambda}{d}}(\R^d)$ \cite[Theorem 2.1.6]{Gr}.
\end{proof}


\subsection{Relaxation to density matrices}

We use the result from the previous subsection to make the next step towards \eqref{eq:upper}, namely by proving an upper bound in terms of density matrices.

We recall that a nonnegative trace class operator $\gamma$ on $L^2(\R^d)$ has a well defined density $\rho_\gamma\in L^1(\R^d)$. Indeed, if we decompose $\gamma= \sum_i \lambda_i |\psi_i\rangle\langle\psi_i|$ with orthonormal $\psi_i$, then $\rho_\gamma = \sum_i \lambda_i |\psi_i|^2$. (In the case of a non-simple eigenvalue $\lambda_i$ one can convince oneself easily that this is independent of the choice of the eigenfunction $\psi_i$.)

We claim that for any operator $\gamma$ on $L^2(\R^d)$ satisfying
\begin{equation}
	\label{eq:gammaproperties}
	0\leq\gamma\leq 1
	\qquad\text{and}\qquad
	N:= \Tr\gamma \in\N
\end{equation}
we have
\begin{equation}
	\label{eq:constdensitymatrix}
	\Tr(-\Delta)^s \gamma \geq \kappa_N^{(d,s)} \left( D_{2s}[\rho_\gamma] - \const \int_{\R^d} \rho_\gamma(x)^{1+\frac{2s}d}\,dx \right).
\end{equation}
Here, as usual, we write $\Tr(-\Delta)^s \gamma$ instead of $\Tr(-\Delta)^\frac s2 \gamma (-\Delta)^\frac s2$. Of course the bound is only meaningful if the latter quantity is finite.

Given Proposition \ref{indirect} the proof of this assertion is relatively standard (see, e.g., \cite[Section 3]{LiTh84}), but we include some details for the sake of completeness. First, there exists a nonnegative operator $\Gamma$ on the antisymmetric subspace of $L^2(\R^{dN})$ satisfying
$$
\Tr_{N-1} \Gamma = \gamma  \,,
$$
where $\Tr_{N-1}$ denotes the partial trace with respect to $N-1$ variables. This is due to \cite{Co}; see also \cite[Theorem 3.2]{LiSe}. It follows that
$$
\Tr\Gamma = N^{-1} \Tr\gamma = 1
\qquad\text{and}\qquad
\sum_{n=1}^N \Tr(-\Delta_n)^s\Gamma = \Tr(-\Delta)^s\gamma \,.
$$
Expanding
$$
\Gamma = \sum_i p_i |u_i\rangle\langle u_i|
$$
with orthonormal antisymmetric functions $u_i\in L^2(\R^{dN})$ and nonnegative numbers $p_i$ we obtain
$$
\sum_i p_i = 1
\qquad\text{and}\qquad
\sum_i p_i \sum_{n=1}^N \int_{\R^{dN}} |(-\Delta_n)^{\frac s2}u_i|^2\,dX = \Tr(-\Delta)^s\gamma \,.
$$
Therefore, the definition of $\kappa_N^{(d,s)}$, applied to $u_i$, yields the inequality
$$
\Tr(-\Delta)^s\gamma \geq \kappa_N^{(d,s)} \sum_i p_i \sum_{n<m} \int_{\R^{dN}} \frac{|u_i(X)|^2}{|X_n-X_m|^{2s}}\,dX \,.
$$
We now apply Proposition \ref{indirect} to the measure
$$
d\mu(X) = \sum_i p_i |u_i(X)|^2 \,dX \,,
$$
which, by the above properties, is indeed a probability measure. Moreover, using the partial trace relation between $\Gamma$ and $\gamma$, we find $\rho_\mu=\rho_\gamma$. Thus, the claimed inequality \eqref{eq:constdensitymatrix} follows from Proposition \ref{indirect}.


\subsection{Construction of $\gamma$ using coherent states}

Let $0\leq\rho\in L^1\cap L^{1+\frac{2s}d}(\R^d)$ with $\int_{\R^d} \rho\,dx=N$ and let $g\in H^s(\R^d)$ be $L^2$-normalized. We consider the operator
$$
\gamma(x,x') = \iint_{\R^d\times\R^d} g(y-x) \, e^{i\eta\cdot(x-x')} \, \1(|\eta|^{2s}<c \rho(y)^\frac{2s}{d}) \, \overline{g(y-x')} \, \frac{dy\,d\eta}{(2\pi)^d}
$$
with $c = (2\pi)^{2s} \omega_d^{-\frac{2s}{d}}$, where $\omega_d$ is the volume of the unit ball in $\R^d$.

It is easy to see that this operator satisfies \eqref{eq:gammaproperties}. Indeed, the bound $\gamma\geq 0$ follows immediately by estimating $\1(|\xi|^{2s}<c \rho(x)^\frac{2s}{d})\geq 0$ and the bound $\gamma\leq 1$ follow by estimating $\1(|\xi|^{2s}<c \rho(x)^\frac{2s}{d})\leq 1$ and using Plancherel and the normalization of $g$. To prove $\Tr\gamma=N$ we integrate the kernel on the diagonal, using the choice of $c$ and, again, the normalization of $g$. In this connection we also note that the density of $\gamma$ is
$$
\rho_\gamma(x) = \int_{\R^d} \rho(x) |g(y-x)|^2 \,dy = \rho*|g|^2(x) \,.
$$
Assuming that $|\widehat g|$ is even, we claim that
$$
\Tr(-\Delta)^s\gamma \leq c_{d,s}^{\rm TF} \int_{\R^d} (\rho*|g|^2)^{1+\frac{2s}d}\,dx + N \|(-\Delta)^\frac s2 g\|_2^2 \,.
$$
This is shown in the special case $d=3$, $s=\frac12$ in \cite[Section 3]{LiTh84} (see also \cite[Theorem 12.10]{LiLo}). The proof generalizes to the general case, the underlying estimates being the same as in the proof of Lemma \ref{coherent}.

If we insert these facts into \eqref{eq:constdensitymatrix}, we obtain
\begin{align*}
	& c_{d,s}^{\rm TF} \int_{\R^d} (\rho*|g|^2)^{1+\frac{2s}d}\,dx + N \|(-\Delta)^\frac s2 g\|_2^2 \\
	& \geq \kappa_N^{(d,s)} \left( D_{2s}[\rho*|g|^2] - \const \int_{\R^d} (\rho*|g|^2)^{1+\frac{2s}d}\,dx \right).
\end{align*}
By the normalization of $g$ and Minkowski's inequality, we have
$$
\int_{\R^d} (\rho*|g|^2)^{1+\frac{2s}d}\,dx \leq \int_{\R^d} \rho^{1+\frac{2s}d}\,dx \,.
$$
Moreover, as in the proof of \eqref{eq:lower}, Young's convolution inequality shows that
$$
D_{2s}[\rho*|g|^2] \geq D_{2s}[\rho] - \frac12 \|\rho\|_{1+\frac{2s}d}^2 \left\| |x|^{-2s} - |g|^2 * |x|^{-2s}*|g|^2 \right\|_{\frac{d+2s}{4s}} \,.
$$
To summarize, we have
$$
c^{\rm TF}_{d,s} \int_{\R^d} \rho^{1+\frac{2s}d}\,dx + N \|(-\Delta)^\frac s2 g\|_2^2 
\geq \kappa_N^{(d,s)} \left( D_{2s}[\rho] - \mathcal R \right)
$$
with
$$
\mathcal R := \frac12 \|\rho\|_{1+\frac{2s}d}^2 \left\| |x|^{-2s} - |g|^2 * |x|^{-2s}*|g|^2 \right\|_{\frac{d+2s}{4s}} + \const \int_{\R^d} \rho^{1+\frac{2s}d}\,dx \,.
$$
Similarly as in the proof of the lower bound we now assume that $g(x) = \ell^{-\frac d2} G( \ell^{-1} x)$ for an $L^2$-normalized function $G\in H^s(\R^d)$ and a parameter $\ell>0$ to be chosen. We consider $G$ as fixed and obtain, as before
$$
\mathcal R \lesssim \ell^{\frac{2s(d-2s)}{d+2s}} \|\rho\|_{1+\frac{2s}d}^2 + \|\rho\|_{1+\frac{2s}d}^{1+\frac{2s}d} \,.
$$
Thus,
\begin{equation}
	\label{eq:upperboundproof}
	c^{\rm TF}_{d,s} \int_{\R^d} \rho^{1+\frac{2s}d}\,dx + \const \ell^{-2s} N 
	\geq \kappa_N^{(d,s)} \left( D_{2s}[\rho] - \const \left( \ell^{\frac{2s(d-2s)}{d+2s}} \|\rho\|_{1+\frac{2s}d}^2 + \|\rho\|_{1+\frac{2s}d}^{1+\frac{2s}d} \right) \right).
\end{equation}


\subsection{The semiclassical problem}

The following result states that the variational problem defining $\tau_{d,s}$ has an optimizer. This result is not strictly necessary for our proof of the upper bound in Theorem \ref{main}, but it is readily available and makes the proof more transparent.

\begin{lemma}\label{exopt}
	Let $d\geq 1$ and $0<s<\frac d2$. Then there is a $0\leq\rho_*\in L^{1+\frac{2s}d}\cap L^1(\R^d)$, $\rho_*\neq 0$, such that
	$$
	\frac{\int_{\R^d} \rho_*(x)^{1+\frac{2s}d}\,dx \left( \int_{\R^d} \rho_*(x)\,dx \right)^{1-\frac{2s}d}}{D_{2s}[\rho_*]}
	= \tau_{d,s} \,.
	$$
\end{lemma}

In the special case $s=\frac12$, $d=3$ this appears in \cite[Appendix A]{LiOx}. The proof in the general case is exactly the same.

For the sake of completeness we mention that the uniqueness (up to translations, dilations and multiplication by a constant) of $\rho_*$ has been studied in \cite{LiYa1}, as well as in the recent papers \cite{CaCaHo1,CaCaHo2}.


\subsection{Proof of the upper bound in Theorem \ref{main}}

Let $\rho_*$ be the optimizer from Lemma \ref{exopt}. After a dilation and a multiplication by a constant we may assume that
$$
\int_{\R^d} \rho_*\,dx = 1 = \int_{\R^d} \rho_*^{1+\frac{2s}d}\,dx \,,
\qquad
D_{2s}[\rho_*] = \tau_{d,s}^{-1} \,.
$$
We then apply the construction outlined in this section with the choice $\rho = N \rho_*$. Inequality \eqref{eq:upperboundproof} turns into
$$
c_{d,s}^{\rm TF} \left( 1 + \const \ell^{-2s} N^{-\frac{2s}d} \right) \geq \kappa_N^{(d,s)} \tau_{d,s}^{-1} N^{1-\frac{2s}d} \left( 1 - \const \left( \ell^{\frac{2s(d-2s)}{d+2s}} + N^{-1+\frac{2s}d} \right) \right).
$$
Choosing
$$
\ell = N^{-\frac{d+2s}{2d^2}}
$$
we obtain, for all sufficiently large $N$, the claimed bound \eqref{eq:upper}.


\appendix

\section{An order of magnitude bound}

Our goal in this appendix is to prove the lower bound
\begin{equation}
	\label{eq:lowerordersharp}
	\inf_{N\geq 2} N^{1-\frac{2s}d} \kappa_N^{(d,s)} > 0
\end{equation}
for $0<s<\frac d2$ with $s\leq 1$. This is weaker than the asymptotics in Theorem \ref{main}, but it does capture the right order of magnitude as $N\to\infty$ and we feel that the argument is robust and may be useful in other contexts as well.

The main step in the proof of \eqref{eq:lowerordersharp} is the following bound, which is similar to the sought-after Hardy inequality, but with an additional positive term on the left side.

\begin{proposition}\label{hardyprop}
	Let $d\geq 1$, $0<s<\frac d2$ and $\tau>0$. Then there is a constant $C(\tau)>0$ such that for any $N\geq 2$ and any antisymmetric function $u\in \dot H^s(\R^{dN})$,
	\begin{equation*}
		\sum_{n=1}^N \int_{\R^{dN}} \left(|(-\Delta_n)^\frac s2 u|^2+\frac{\tau |u|^2}{\delta_n^2}\right) \,dX
		\geq C(\tau) N^{-1+\frac{2s}d} \sum_{1\leq n<m\leq N} \int_{\R^{dN}} \frac{|u|^2}{|X_n-X_m|^{2s}} \,dX \,.
	\end{equation*}
\end{proposition}

\begin{proof}[Proof of \eqref{eq:lowerordersharp} given Proposition \ref{hardyprop}]
	Denoting by $c$ the implicit constant in Proposition \ref{nearestneighbor}, we infer from that proposition and from Proposition \ref{hardyprop} that for any $0<\theta<1$
	\begin{align*}
		\sum_{n=1}^N (-\Delta_n)^\frac s2 
		& \geq (1-\theta) \sum_{n=1}^N \left( (-\Delta_n)^s + \frac{\theta\, c}{1-\theta} \, \delta_n^{-2s} \right) \\
		& \geq (1-\theta)\, C( \tfrac{\theta\,c}{1-\theta}) \, N^{-1+\frac{2s}d} \sum_{1\leq n<m\leq N} |X_n-X_m|^{-2s}.
	\end{align*}
	Hence $N^{1-\frac{2s}d} \kappa_N \geq \sup_{0<\theta<1} (1-\theta)\, C( \tfrac{\theta\,c}{1-\theta})>0$, as claimed.
\end{proof}

We emphasize that, while Proposition \ref{hardyprop} does not require $s\leq 1$, our proof of \eqref{eq:lowerordersharp} does, since we apply Proposition \ref{nearestneighbor}.

It remains to prove Proposition \ref{hardyprop} and to do so, we proceed again with the help of the L\'evy-Leblond method \cite{LL}. We split the $N$ variables $X=(X_1,\ldots,X_N)$ into a group of `electronic' variables $Y=(Y_1,\ldots,Y_M)$ and a group of `nuclear' variables $R=(R_1,\ldots,R_K)$ with $M+K=N$. For a fixed $R\in\R^{dK}$ with $R_k\not =R_l$ for $k\not =l$ we define the function on $\R^d$,
\begin{equation*}
	V_R(y) := \sum_{k=1}^M \frac1{|y-R_k|^{2s}} - \frac1{\delta_R(y)^{2s}} \,,
\end{equation*} 
and the constant
\begin{equation*}
	U_R := \sum_{1\leq k<l\leq K} \frac1{|R_k-R_l|^{2s}} + \sum_{k=1}^K \frac1{\delta_k(R)^{2s}} \,.
\end{equation*}
Here, as before, $\delta_R(y)=\min\{|y -R_k| :\, 1\leq k\leq K \}$ and $\delta_k(R)=\min\{|R_k - R_l|:\ 1\leq k\leq K\,,\ l\neq k \}$.

We will estimate the sum of the negative eigenvalues of the (one-particle) operator $(-\Delta)^s -\lambda V_R$ in $L^2(\R^d)$ in terms of $U_R$.

\begin{lemma}\label{ltvu}
	Let $d\geq 1$ and $0<s<\frac d2$. Then, for all $K\geq 2$, $R\in\R^{dK}$ and $\lambda>0$,
	\begin{equation}\label{eq:lt}
		\Tr((-\Delta)^s -\lambda V_R)_- \lesssim \lambda^{1+ \frac d{2s}} \, K^{\frac d{2s} - 1} \, U_R
	\end{equation}
	with an implicit constant that only depends on $d$ and $s$.
\end{lemma}

\begin{proof}
	By the Lieb-Thirring inequality (see, e.g., \cite[Theorem 4.60]{FrLaWe}) we have
	\begin{equation*}
		\Tr((-\Delta)^s -\lambda V_R)_- \lesssim
		\lambda^{1 + \frac d{2s}} \int_{\R^d} V_R(y)^{1+\frac d{2s}} \,dy \,.
	\end{equation*}
	To estimate the latter integral we write $V_R(y) = \sum_{k=1}^K \chi_k(y) |y-R_k|^{-2s}$ where $1-\chi_k$ is the characteristic function of the Voronoi cell $\Gamma_k := \{ y:\ |y-R_k|=\min_l |y-R_l| \}$. Note that H\"older's inequality implies
	\begin{equation*}
		V_R(y)^{\frac12 (1+ \frac{d}{2s})} \leq K^{\frac{d-2s}{4s}} \sum_{k=1}^K \chi_k(y) |y-R_k|^{-\frac{d+2s}2} \,.
	\end{equation*}
	Hence
	\begin{align*}
		\int_{\R^d} V_R(y)^{1+\frac d{2s}} \,dy 
		& \leq K^\frac{d-2s}{2s} \sum_{k,l} \int_{\R^d} \chi_k(y) \chi_l(y) |y-R_k|^{-\frac{d+2s}2} |y-R_l|^{-\frac{d+2s}2} \,dy \\
		& \leq K^\frac{d-2s}{2s} \left(2 \sum_{k<l} \int_{\R^d} |y-R_k|^{-\frac{d+2s}2} |y-R_l|^{-\frac{d+2s}2} \,dy \right. \\
		& \qquad\qquad\ \ \left. + \sum_{k} \int_{\R^d} \chi_k(y) |y-R_k|^{-d-2s} \,dy \right).
	\end{align*}
	The first integral is easily found to equal a constant times $|R_k-R_l|^{-2s}$. To estimate the second integral we note that $\{ y:\ |y-R_k|\leq \delta_k(R)/2 \} \subset \Gamma_k$. Extending the domain of integration we find
	\begin{equation*}
		\int_{\R^d} \chi_k(y) |y-R_k|^{-d-2s} \,dy 
		\leq \int_{\{|y-R_k|> \delta_k(R)/2\}} |y-R_k|^{-d-2s} \,dy 
		= \const \delta_k(R)^{-2s}.
	\end{equation*}
	This proves the assertion.
\end{proof}

Now everything is in place for the

\begin{proof}[Proof of Proposition \ref{hardyprop}]
	In view of \eqref{eq:tauspos}, it suffices to prove the bound for sufficiently large $N$. In fact, we will prove the bound for $N\geq N(\tau)$ for some $N(\tau)$ to be determined later.
	
	For given $N\geq 3$, $\tau>0$ and $\kappa>0$ we choose an integer $M\in\{1,\ldots,N-2\}$ and parameters $\lambda,\alpha>0$. Setting $K:=N-M$ we write
	\begin{equation}\label{eq:decomp}
		\sum_{n=1}^N \left((-\Delta_n)^s + \tau\delta_n^{-2s}\right) - \kappa\sum_{1\leq n<m\leq N} |X_n-X_m|^{-2s}
		= \frac NM {N \choose M}^{-1} \sum_\pi h_\pi \,.
	\end{equation}
	Here the sum runs over all partitions $\pi=(\pi_1,\pi_2)$ of $\{1,\ldots,N\}$ into two disjoint sets $\pi_1$, $\pi_2$ of sizes $M$ and $K$, respectively, and for any such partition the operator $h_\pi$ is defined by
	\begin{align*}
		h_\pi & := 
		\sum_{m\in\pi_1}\left((-\Delta_m)^s - \lambda\sum_{k\in\pi_2} |X_m-X_k|^{-2s} + \lambda \delta_m^{-2s} \right) \\
		& \quad + \alpha\sum_{k<l\in\pi_2} |X_k-X_l|^{-2s} 
		+ \alpha\sum_{k\in\pi_2}\delta_k^{-2s}.
	\end{align*}
	In order that \eqref{eq:decomp} be an identity we require
	\begin{equation}\label{eq:req1}
		\lambda M + \alpha K = \tau M
	\end{equation}
	and
	\begin{equation}\label{eq:req2}
		2\lambda MK - \alpha K(K-1) = \kappa M(N-1) \,.
	\end{equation}
	
	It suffices to prove that for $\kappa\leq C(\tau) N^{-1+\frac{2s}d}$ one has $h_\pi\geq 0$ for all partitions $\pi$ as above. We denote the variables in $\pi_1$ by $Y=(Y_1,\ldots,Y_M)$ and those in $\pi_2$ by $R=(R_1,\ldots,R_K)$. Then one has the estimates
	\begin{equation*}
		\delta_m(x) \geq \delta_R(Y_j),
		\qquad m\in\pi_1
	\end{equation*}
	and
	\begin{equation*}
		\delta_k(x) \geq \delta_k(R),
		\qquad k\in\pi_2 \,.
	\end{equation*}
	These two estimates lead to the lower bound
	\begin{equation}\label{eq:fixedr}
		h_\pi \geq
		\sum_{m=1}^M \left( (-\Delta_{Y_m})^s - \lambda V_R(Y_j)\right) 
		+ \alpha U_R \,.
	\end{equation}
	The right side is an operator in $L^2(\R^{dN})$, but there is no kinetic energy associated with the $R$ variables. Hence if we define for fixed $R\in\R^{dK}$ an operator $h^R$ in the antisymmetric subspace of $L^2(\R^{dM})$ by the expression on the right side of \eqref{eq:fixedr}, then one has the estimate
	\begin{equation*}
		h_\pi \geq \inf_{R\in\R^{dK}} \inf\spec h^R.
	\end{equation*}
	Further, since $h^R$ acts on antisymmetric functions one has 
	$$
	\inf\spec h^R \geq-\tr(-\Delta-\lambda V_R)_- + \alpha U_R \,,
	$$
	and hence by Lemma \ref{ltvu}
	\begin{equation*}
		\inf_{R\in\R^{dK}}\inf\spec h^R \geq 0
	\end{equation*}
	provided
	\begin{equation}\label{eq:req3}
		\alpha-C \lambda^{\frac d{2s} + 1} K^{\frac d{2s}-1} \geq 0.
	\end{equation}
	
	It remains to choose the parameters $M,\alpha,\lambda$ such that \eqref{eq:req1}, \eqref{eq:req2} and \eqref{eq:req3} are satisfied. With the choice $\alpha/\lambda=M/K$ equation \eqref{eq:req2} becomes $\lambda = \frac\tau 2$, \eqref{eq:req1} becomes
	\begin{equation}\label{eq:kappa}
		\kappa = \frac{\tau (K+1)}{2(N-1)}
	\end{equation}
	and \eqref{eq:req3} becomes
	\begin{equation}
		\label{eq:req32}
		M \geq C 2^{-\frac d{2s}} \, \tau^{\frac d{2s}} K^{\frac d{2s}} \,.
	\end{equation}
	
	We choose $K= [\epsilon \tau^{-1} N^\frac{2s}d]$, where $\epsilon>0$ will be determined below (depending only on $d$ and $s$). As we mentioned at the beginning of the proof we may assume that $N^{1-\frac{2s}d} \geq 2\epsilon \tau^{-1} =:N(\tau)^{1-\frac{2s}d}$, which guarantees that $K\leq\frac N2$ and consequently $M\geq\frac N2$. This implies that \eqref{eq:req32} is satisfied, provided $\epsilon>0$ is chosen small enough depending on $d$ and $s$. Then $\kappa$ given by \eqref{eq:kappa} is easily seen to satisfy $\kappa\leq C(\tau) N^{-1+\frac{2s}d}$ for all $N\geq N(\tau)$. This completes the proof of Proposition \ref{hardyprop}.
\end{proof}


\section{The borderline case $s=1$, $d=2$}\label{sec:upper2d}

In this appendix, for the sake of definiteness, we focus on the case $s=1$. Our main result assumes $d\geq 3$ and its proof breaks down in several places in dimensions $d=1,2$. Meanwhile, for $d=1$ we know from \cite{HO2LaTi} that $\kappa_N^{(1)}=\frac 12$ for all $N$. In particular, this constant is independent of $N$. In the remaining case $d=2$, we only know that $\kappa_N^{(2)}\geq 4N^{-1}$ for all $N$, but this does probably not capture the correct large $N$-behavior. The following result gives an upper bound

\begin{proposition}
	Let $d=2$ and $s=1$. Then
	\begin{equation}
		\label{eq:upper2d}
		\limsup_{N\to\infty} \ (\ln N)\, \kappa_N^{(2)} \leq 4 \,.
	\end{equation}
\end{proposition}

It is a tantalizing question whether the right side of \eqref{eq:upper2d} is, in fact, the limit of $(\ln N)\, \kappa_N^{(2)}$. We would like to express our gratitude to Robert Seiringer for first suggesting \eqref{eq:upper2d} and for several discussions related to it.

\begin{proof}
Our construction depends on two main parameters, $L$ and $\mu$. Given a sequence of $N$'s tending to infinity, these parameters will be chosen such that $N=\#\mathcal N$ for a certain set $\mathcal N$ satisfying
$$
\left\{ p\in \frac{2\pi}{L}\,\Z^2 :\ |p|^2 < \mu \right\} \subset \mathcal N \subset \left\{ p\in \frac{2\pi}{L}\,\Z^2 :\ |p|^2 \leq \mu \right\}.
$$
This implies that
$$
\frac1{4\pi}\, \mu L^2 \sim N \to \infty \,.
$$
The antisymmetric function $u$ on $\R^{2N}$ that we will use as a trial function to bound $\kappa_N^{(2)}$ from above will be a Slater determinant of functions that are essentially plane wave restricted to $Q_L:=(-L/2,L/2)^2$ with momenta in $\mathcal N$. There are several ways to construct such functions and here we use a method that we learned from \cite{GoLeNa}.

Let $0\leq\zeta\in C^1_c(Q_\ell)$ with $\int_{Q_\ell} \zeta\,dx = 1$, where $\ell>0$ is a parameter satisfying $\ell\ll L$. (We keep track of it only for dimensional consistency.) For $p\in \frac{2\pi}{L}\,\Z^2$ let
$$
\phi_p(x) := L^{-1} \, \sqrt{ \1_{Q_L} * \zeta }\, e^{i p\cdot x}
\qquad\text{for all}\ x\in\R^2 \,.
$$
A computation \cite{GoLeNa} (see also \cite[Lemma 7.21]{FrLaWe}), based on the Fourier transform of the characteristic function of an interval, shows that the $\phi_p$ are orthonormal in $L^2(\R^2)$. We define $u$ as their Slater determinant,
$$
u(X) = (N!)^{-\frac12}\, \det \left( \phi_{p_n}(X_{n'}) \right)_{p_n,p_{n'}\in\mathcal N} \,,
$$
where $p_1,\ldots, p_N$ is an enumeration of $\mathcal N$. We have  \cite{GoLeNa} (see also \cite[Lemma 7.21]{FrLaWe})
\begin{equation}
	\label{eq:upper2dkinetic}
	\sum_{n=1}^N \int_{\R^{2N}} |\nabla_n u|^2\,dX = \sum_{p\in\mathcal N} \int_{\R^2} |\nabla\phi_p|^2\,dx = \frac1{8\pi} \, \mu^2 L^2 \left( 1+ o(1) \right)
\end{equation}
in the asymptotic regime that we are considering.

Our task is to bound from below
\begin{equation}
	\label{eq:upper2dinter}
	\sum_{1\leq n<m\leq N} \int_{\R^{2N}} \frac{|u(X)|^2}{|X_n-X_m|^2}\,dX = \frac12 \iint_{\R^2\times\R^2} \frac{\rho_u(x)\,\rho_u(x') - |\gamma_u(x,x')|^2}{|x-x'|^2}\,dx\,dx \,,
\end{equation}
where
\begin{align*}
	\rho_u(x) & = \sum_{p\in\mathcal N} |\phi_p(x)|^2 = L^{-2} \, N \, \1_{Q_L} * \zeta \,,\\
	\gamma_u(x,x') & = \sum_{p\in\mathcal N} \phi_p(x) \overline{\phi_p(x')} = L^{-2} \sqrt{1_{Q_L}*\zeta(x)} \, \sqrt{1_{Q_L}*\zeta(x')} \sum_{p\in\mathcal N} e^{ip\cdot(x-x')} \,.
\end{align*}
Note that the integrand on the right side of \eqref{eq:upper2dinter} is nonnegative. Consequently, we obtain a lower bound by restricting it to
$$
\Omega := \{ (x,x') \in Q_{L-\ell} \times Q_{L-\ell} :\ \sqrt\mu\, |x-x'| > C \}
$$
for a certain constant $C$, independent of $L$ and $\mu$, and to be chosen below. Note that for $x\in Q_{L-\ell}$ we have $1_{Q_L}*\zeta(x) = 1$. Therefore the $\rho$-part of the integral on the right side of \eqref{eq:upper2dinter}, restricted to $\Omega$, is bounded from below by
$$
L^{-4} N^2 \iint_{Q_{L-\ell} \times Q_{L-\ell}} \frac{\1(\sqrt\mu\, |x-x'| > C)}{|x-x'|^2}\,dx\,dx' = L^{-4} N^2 (L-\ell)^2 \, \mathcal I(C\mu^{-\frac12}(L-\ell)^{-1})
$$
where
$$
\mathcal I(\epsilon) := \iint_{Q_1\times Q_1} \frac{\1(|y-y'|\geq \epsilon)}{|y-y'|^2}\,dy\,dy' \,.
$$
An elementary computation shows that $\mathcal I(\epsilon) = 2\pi (\ln\frac1\epsilon) (1+o(1))$. Using $\ell\ll L$ and $N\sim (4\pi)^{-1} \mu L^2$, we deduce that
$$
\iint_\Omega \frac{\rho_u(x)\,\rho_u(x')}{|x-x'|^2}\,dx\,dx' = 2\pi \, L^{-2} N^2 (\ln(\sqrt\mu L)) \, (1+o(1)) = \frac1{16\pi} \, \mu^2 L^2 (\ln N) \, (1+o(1)).
$$
Comparing this with \eqref{eq:upper2dkinetic} we arrive at the constant $4$ on the right side of \eqref{eq:upper2d}.

Thus it remains to prove that the $\gamma$-part of the right side of \eqref{eq:upper2dinter}, restricted to $\Omega$, is negligible for some (and, in fact, any) choice of $C$. Before giving a complete proof, let us explain the heuristics. The nonrigorous step is that we approximate
$$
L^{-2} \sum_{p\in\mathcal N} e^{ip\cdot(x-x')} \approx \int_{|p|^2<\mu} e^{ip\cdot(x-x')}\,\frac{dp}{(2\pi)^2} \,.
$$
(Since $\mu L^2\gg 1$, this is justified for fixed $x-x'$, but we shall use it uniformly for $\sqrt\mu |x-x'|\geq C>0$ with $|x - x'|\leq L$.) It is known that
$$
\int_{|p|^2<\mu} e^{ip\cdot(x-x')}\,\frac{dp}{(2\pi)^2} = (2\pi)^{-1} \, \sqrt\mu \, |x-x'|^{-1} \, J_1(\sqrt\mu |x-x'|) \,,
$$
where $J_1$ is a Bessel function \cite[Chapter 9]{AbSt}. Using the decay bound on Bessel functions, $|J_1(t)|\lesssim t^{-1/2}$, \cite[(9.2.1)]{AbSt} we obtain
$$
\left| \int_{|p|^2<\mu} e^{ip\cdot(x-x')}\,\frac{dp}{(2\pi)^2} \right| \lesssim \mu^\frac14\, |x-x'|^{-\frac 32} \,.
$$
From this we arrive at the expectation that for $x,x'\in Q_{L-\ell}$, at least on average, one has
\begin{equation}
	\label{eq:upper2dexchangebound}
	|\gamma_u(x,x')|  \lesssim \mu^\frac14\, |x-x'|^{-\frac 32} \,.
\end{equation}
Accepting this bound, we obtain by straightforward estimates
$$
\iint_\Omega \frac{|\gamma_u(x,x')|^2}{|x-x'|^2}\,dx\,dx' \lesssim \sqrt\mu \iint_\Omega \frac{dx\,dx'}{|x-x'|^5} \lesssim \mu^2 L^2 \,.
$$
Recalling that our lower bound on the $\rho$-term in \eqref{eq:upper2dinter} is of size $\mu^2 L^2 \ln (\mu L^2)$, we see that the $\gamma$-term is indeed negligible. 

We now present a rigorous proof that the $\gamma$-term is neglibigle. We will not be able to prove the bound in \eqref{eq:upper2dexchangebound}, but we will be able to prove that
\begin{equation}
	\label{eq:upper2dexchangebound2}
	|\gamma_u(x,x')|  \lesssim \mu^\frac12\, |x-x'|^{-1} 
	\qquad\text{if}\ (x,x')\in Q_{L-\ell} \times Q_{L-\ell} \,.
\end{equation}
Accepting this bound and combining it with the trivial bound $|\gamma_u(x,x')|^2\leq \rho_u(x)\rho_u(x')$, we obtain
$$
\iint_\Omega \frac{|\gamma_u(x,x')|^2}{|x-x'|^2}\,dx\,dx' \lesssim 
L^2 \left( \mu \int_{C\mu^{-\frac12}< |r|\leq \frac12 L} \frac{dr}{|r|^4} + \mu^2 \int_{\frac12 L< |r|< L} \frac{dr}{|r|^2} \right) \lesssim \mu^2 L^2 \,,
$$
which is the same as if the heuristic bound \eqref{eq:upper2dexchangebound} was true.

It remains to prove \eqref{eq:upper2dexchangebound2}. We bound
$$
\left| \sum_{p\in\mathcal N} e^{ip\cdot r} \right| \leq \min\left\{ \sum_{p_1^2\leq\mu} \left| \sum_{p_2:\, (p_1,p_2)\in\mathcal N} e^{ip_2r_2} \right|, \sum_{p_2^2\leq\mu} \left| \sum_{p_1:\, (p_1,p_2)\in\mathcal N} e^{ip_1r_1} \right| \right\}
$$
and use the elementary inequality, valid for any interval $I\subset\R$ and $t\in\R\setminus L \Z$,
$$
\left| \sum_{\tau\in I\cap \frac{2\pi}{L}\Z} e^{i\tau t} \right| \lesssim \frac L{\dist(t,L\Z)} \,.
$$
The latter follows by summing a trigonometric series. Since the sum over $p_j^2\leq\mu$ contains $\lesssim \mu^\frac12 L$ elements, we deduce that
$$
\left| \sum_{p\in\mathcal N} e^{ip\cdot r} \right| \lesssim \mu^\frac12 L \min\left\{ \frac L{\dist(r_2,L\Z)}, \frac L{\dist(r_1,L\Z)} \right\}.
$$
If $|r|\leq \frac L2$, then $\dist(r_j,L\Z)=|r_j|$ for $j=1,2$. Moreover, $\max_j r_j^2 \geq \frac12 |r|^2$. Therefore the right side is $\lesssim \mu^\frac12 L^2 |r|^{-1}$. This, applied to $r=x-x'$, yields \eqref{eq:upper2dexchangebound2} and concludes the proof of the proposition.
\end{proof}

\begin{remarks}
	(a) In physics, the vanishing of $\rho_u(x)\,\rho_u(x') - |\gamma_u(x,x')|^2$ near $x=x'$ is called the \emph{exchange hole}. The intuition is that it is this exchange hole that leads to the Hardy inequality for (spin-polarized) fermions in two dimensions. This hole, which is of size $\mu^{-\frac12}$ in the above example, mitigates the logarithmic divergence of the integral $|x-x'|^{-2}$ and leads to the logarithmic behavior of the constant $\kappa_N^{(2)}$.\\
	(b) It is essential for the validity of $\kappa_N^{(2)}>0$ that the fermionic particles are spinless (or spin-polarized). If $u$ has two or more spin states, there is no reason that $\sum_{n<m} \sum_\sigma \iint_{\R^{2N}} |X_n-X_m|^{-2} |u(X,\sigma)|^2\,dX$ is finite. (Here $\sum_\sigma$ denotes the sum over spin-states.)  
\end{remarks}


\bibliographystyle{amsalpha}

\begin{thebibliography}{39}

\bibitem{AbSt} M. Abramowitz, I. Stegun, \textit{Handbook of mathematical functions with formulas, graphs, and mathematical tables}. Reprint of the 1972 ed. A Wiley-Interscience Publication. Selected Government Publications. New York: John Wiley \& Sons, Inc; Washington, D.C.: National Bureau of Standards. (1984).

\bibitem{CaCaHo1} V. Calvez, J. A. Carrillo, F. Hoffmann, \textit{The geometry of diffusing and self-attracting particles in a one-dimensional fair-competition regime}. In:
Bonforte, Matteo (ed.) et al., Nonlocal and nonlinear diffusions and interactions: new methods and directions. Lecture Notes in Mathematics 2186, 1--71 (2017).

\bibitem{CaCaHo2} V. Calvez, J. A. Carrillo, F. Hoffmann, \textit{Uniqueness of stationary states for singular Keller-Segel type models}. Nonlinear Anal., Theory Methods Appl., Ser. A, Theory Methods \textbf{205} (2021), Article ID 112222, 25 p.

\bibitem{Co} A. J. Coleman, \textit{Structure of fermion density matrices}. Rev. Mod. Phys. \textbf{35} (1963), 668--686.

\bibitem{Con} J. Conlon, \textit{The ground state energy of a classical gas}. Commun. Math. Phys. \textbf{94} (1984), no. 4, 439--458.

\bibitem{Da} E. B. Davies, \textit{A review of Hardy inequalities}. Oper. Theory Adv. Appl., 110, Birkh\"auser Verlag, Basel, 1999, 55--67.

\bibitem{FedlL} C. Fefferman, R. de la Llave, \textit{Relativistic stability of matter. I}. Rev. Mat. Iberoamericana \textbf{2} (1986), no. 1-2, 119--213.

\bibitem{Fr} R. L. Frank, \textit{The Lieb--Thirring inequalities: recent results and open problems}. In: Kechris, A. (ed.) et al., Nine mathematical challenges. An elucidation.  Providence, RI: American Mathematical Society (AMS). Proc. Symp. Pure Math. \textbf{104} (2021), 45--86.

\bibitem{Fr2} R. L. Frank, \textit{Lieb--Thirring inequalities and other functional inequalities for orthonormal systems}. In: Proc. Int. Cong. Math. 2022, Vol. 5, pp. 3756--3774. EMS Press, Berlin, 2023.

\bibitem{FrHuJeNa} R. L. Frank, D. Hundertmark, M. Jex, P. T. Nam, \textit{The Lieb--Thirring inequality revisited}. J. Eur. Math. Soc. (JEMS) \textbf{23} (2021), no. 8, 2583--2600.

\bibitem{FrLaWe} R. L. Frank, A. Laptev, T. Weidl, \textit{Schr\"odinger operators: eigenvalues and Lieb--Thirring inequalities}. Cambridge Studies in Advanced Mathematics 200. Cambridge University Press, Cambridge, 2023. 

\bibitem{FrSe} R. L. Frank, R. Seiringer, \textit{Nonlinear ground state representations and sharp Hardy inequalities}. J. Funct. Anal. \textbf{255} (2008), no. 12, 3407--3430.

\bibitem{FoLeSo} S. Fournais, M. Lewin, J. P. Solovej, \textit{The semi-classical limit of large fermionic systems}. Calc. Var. Partial Differ. Equ. \textbf{57} (2018), no. 4, Paper no. 105, 42 p.

\bibitem{GoLeNa} D. Gontier, M. Lewin, F. Q. Nazar, \textit{The nonlinear Schr\"odinger equation for orthonormal functions: existence of ground states}. Arch. Ration. Mech. Anal. \textbf{240} (2021), no. 3, 1203--1254.

\bibitem{Gr} L. Grafakos, \textit{Classical Fourier analysis}. 3rd ed. Graduate Texts in Mathematics 249. Springer, New York, NY, 2014.

\bibitem{HaSe} C. Hainzl, R. Seiringer, \textit{General decomposition of radial functions on  $\R^n$ and applications to  $N$-body quantum systems}. Lett. Math. Phys. \textbf{61} (2002), no. 1, 75--84.

\bibitem{He} I. W. Herbst, \textit{Spectral theory of the operator $(p^2+m^2)^{1/2} - Z e^2/ r$}. Commun. Math. Phys. \textbf{53} (1977), 285--294; addendum ibid. \textbf{55} (1977), 316.

\bibitem{HO2} M. Hoffmann-Ostenhof, T. Hoffmann-Ostenhof, \textit{``Schr\"odinger inequalities'' and asymptotic behavior of the electron density of atoms and molecules}. Phys. Rev. A \textbf{16} (1077), 1782--1785.

\bibitem{HO2LaTi} M. Hoffmann-Ostenhof, T. Hoffmann-Ostenhof, A. Laptev, J. Tidblom, \textit{Many-particle Hardy inequalities}. J. Lond. Math. Soc. (2) \textbf{77} (2008), no. 1, 99--114.

\bibitem{KoPeSe} V. F. Kovalenko, M. A. Perel'muter, Ya. A. Semenov, \textit{Schr\"odinger operators with $L^{l/2}_w(\R^l)$-potentials}. J. Math. Phys. \textbf{22} (1981), no. 5, 1033--1044.

\bibitem{Le} G. Leoni, \textit{A first course in Sobolev spaces}. 2nd edition. Grad. Stud. Math., 181. Amer. Math. Soc., Providence, RI, 2017.

\bibitem{LL} J.-M. L\'evy-Leblond, \textit{Nonsaturation of gravitational forces}. J. Math. Phys. \textbf{10} (1969), 806--812.

\bibitem{Li0} E. H. Lieb, \textit{A lower bound for Coulomb energies}. Phys. Lett. A \textbf{70} (1979), 444--446.

\bibitem{Li} E. H. Lieb, \textit{Thomas--Fermi and related theories of atoms and molecules}. Rev. Mod. Phys. \textbf{53} (1981), no. 4, 603--641; erratum: ibid. \textbf{54} (1982), no. 1, 311.

\bibitem{LiLo} E. H. Lieb, M. Loss, \textit{Analysis}. Grad. Stud. Math., 14, Amer. Math. Soc., Providence, RI, 2001.

\bibitem{LiOx} E. H. Lieb, S. Oxford, \textit{Improved lower bound on the indirect Coulomb energy}. Int. J. Quantum Chem. \textbf{19} (1981), 427--439.

\bibitem{LiSe} E. H. Lieb, R. Seiringer, \textit{The stability of matter in quantum mechanics}. Cambridge University Press, Cambridge, 2009.

\bibitem{LiSi} E. H. Lieb, B. Simon, \textit{The Thomas--Fermi theory of atoms, molecules and solids}. Adv. Math. \textbf{23} (1977), 22--116.

\bibitem{LiSoYn} E. H. Lieb, J. P. Solovej, J. Yngvason, \textit{Ground states of large quantum dots in magnetic fields}. Phys. Rev. B \textbf{51} (1995), 10646--10665.

\bibitem{LiTh} E. H. Lieb, W. E. Thirring, \textit{Inequalities for the moments of the eigenvalues of the Schr\"odinger Hamiltonian and their relation to Sobolev inequalities}. Stud. math. Phys., Essays Honor Valentine Bargmann, 269--303 (1976).

\bibitem{LiTh84} E. H. Lieb, W. E. Thirring, \textit{Gravitational collapse in quantum mechanics with relativistic kinetic energy}. Ann. Physics \textbf{155} (1984), no. 2, 494--512.

\bibitem{LiYa1} E. H. Lieb, H.-T. Yau, \textit{The Chandrasekhar theory of stellar collapse as the limit of quantum mechanics}. Comm. Math. Phys. \textbf{112} (1987), no. 1, 147--174.

\bibitem{LiYa2} E. H. Lieb, H.-T. Yau, \textit{The stability and instability of relativistic matter}. Comm. Math. Phys. \textbf{118} (1988), no. 2, 177--213.

\bibitem{LuNaPo} D. Lundholm, P. T. Nam, F. Portmann, \textit{Fractional Hardy-Lieb-Thirring and related inequalities for interacting systems}. Arch. Ration. Mech. Anal. \textbf{219} (2016), no. 3, 1343--1382.

\bibitem{Ma} V. Maz'ya, \textit{Sobolev spaces with applications to elliptic partial differential equations}. Grundlehren Math. Wiss., 342, Springer, Heidelberg, 2011.

\bibitem{OpKu} B. Opic, A. Kufner, \textit{Hardy-type inequalities}. Pitman Res. Notes Math. Ser., 219, Longman Scientific \& Technical, Harlow, 1990.

\bibitem{Pe} J. Peetre, \textit{Espaces d'interpolation et th\'eor\`eme de Soboleff}. Ann. Inst. Fourier \textbf{16} (1966), 279--317.

\bibitem{Th} W . Thirring, \textit{A lower bound with the best possible constant for Coulomb Hamiltonians}. Commun. Math. Phys. \textbf{79} (1981), 1--7.

\bibitem{Ya} D. Yafaev, \textit{Sharp constants in the Hardy--Rellich inequalities}. J. Funct. Anal. \textbf{168} (1999), no. 1, 121--144.

\end{thebibliography}

\end{document}